\documentclass[11pt]{amsart}
\usepackage{amsmath,amsfonts,amsthm,amssymb,url,graphicx}

\DeclareFontFamily{OT1}{rsfs}{}
\DeclareFontShape{OT1}{rsfs}{n}{it}{<-> rsfs10}{}
\DeclareMathAlphabet{\mathscr}{OT1}{rsfs}{n}{it}

\DeclareMathOperator{\Prob}{Prob}
\DeclareMathOperator{\supp}{supp}
\DeclareMathOperator{\mo}{\,mod}

\DeclareMathOperator{\diam}{diam}

\DeclareMathOperator{\SL}{SL}
\DeclareMathOperator{\PSL}{PSL}

\DeclareMathOperator{\im}{im}

\DeclareMathOperator{\Sym}{Sym}
\DeclareMathOperator{\Alt}{Alt}

\newtheorem{prop}{Proposition}[section]
\newtheorem{thm}[prop]{Theorem}

\newtheorem{cor}[prop]{Corollary}
\newtheorem{lem}[prop]{Lemma}
\newtheorem{defn}{Definition}

\newtheorem*{defn*}{Definition}

\numberwithin{equation}{section}
\title[Growth in linear algebraic groups and permutation groups]{Growth in linear algebraic groups and permutation groups: towards a unified
  perspective}
\author{H. A. Helfgott}
\address{Harald A. Helfgott, 
  Mathematisches Institut,
Georg-August Universit\"{a}t G\"{o}ttingen, Bunsenstra{\ss}e 3-5, D-37073 G\"{o}ttingen,
Deutschland; IMJ-PRG, UMR 7586,
  58 avenue de France, B\^{a}timent S. Germain, case 7012,
 75013 Paris CEDEX 13, France}
\begin{document}
\begin{abstract}
  By now, we have a product theorem in every finite simple group $G$ of Lie type,
  with the strength of the bound depending only in the rank of $G$. Such theorems
  have numerous consequences: bounds on the diameters of Cayley graphs, spectral
  gaps, and so forth. For the alternating group $\Alt_n$, we have a quasipolylogarithmic
  diameter bound (Helfgott-Seress 2014), but it does not rest on a product theorem.

  We shall revisit the proof of the bound for $\Alt_n$, bringing it closer to the proof
  for linear algebraic groups, and making some common themes clearer. As a result,
  we will show how to prove a product theorem for $\Alt_n$ -- not of full strength,
  as that would be impossible, but strong enough to imply the diameter bound.
%  A further motivation lies in contributing to clear the route towards an eventual
%  strong diameter bound for groups of Lie type and unbounded rank.
  %  Indeed, it is known that we cannot have a product theorem of full strength for
%  $\Alt_n$, or for simple groups of Lie type of unbounded rank. Nevertheless,
%  it is still believed that a polylogarithmic bound on
%  the diameters of the Cayley graphs of both kinds of groups should be true.
    %Our aim is to give a bound better than 
%$\exp(O((\log n)^4))$ for the diameter of any Cayley graph of $\Sym(n)$
%or $\Alt(n)$ (or of any other transitive group on $n$ elements, for that matter).
%We also aim to avoid -- at least for $\Sym(n)$ and $\Alt(n)$ -- any use
%of the Classification Theorem.
\end{abstract}
\maketitle
\section{Introduction}

My personal route in the subject started with the following result.
\begin{thm}[Product theorem \cite{Hel08}]\label{thm:asnat}
  Let $G = \SL_2(\mathbb{Z}/p\mathbb{Z})$, $p$ a prime. Let $A\subset G$ generate $G$.
  Then either
  \[|A\cdot A\cdot A|\geq |A|^{1+\delta}\]
  or
  \[A^k = G,\]
  where $\delta>0$ and $k\in \mathbb{Z}^+$ are absolute constants.\footnote{
    It was soon determined that $k=3$ (\cite{MR2410393}, \cite{BNP}).}
\end{thm}
Here $|S|$ is the number of elements of a set $S$, and $A^k$ denotes
$\{a_1 \dotsc a_k : a_i\in A\}$. (We also write $A B$ for $\{a b:a\in A, b\in B\}$
and $A^{-1}$ for $\{a^{-1} : a\in A\}$.)

Theorem \ref{thm:asnat} gives us an immediate corollary on the diameter of any 
Cayley graph $\Gamma(G,A)$ of $G$.
The {\em diameter} of a graph is the maximal distance $d(v_1,v_2)$ over all pairs
of vertices $v_1$, $v_2$ of a graph $\Gamma$; in turn, the distance $d(v_1,v_2)$ between
two vertices is the length of the shortest path between them, where the length of
a path is defined as its number of edges. In the particular case of a (directed)
Cayley graph
$\Gamma(G,A)$, the diameter equals the least $\ell$ such that every
element $g\in G$ can be expressed as a product of elements of $A$ of length
$\leq \ell$. 
%(Such an $\ell$ exists because $A$ generates $G$ and $G$ is finite.)
\begin{cor}\label{cor:nonat}
  Let $G = \SL_2(\mathbb{Z}/p\mathbb{Z})$, $p$ a prime. Let $S\subset G$ generate $G$.
  Then the diameter of the Cayley graph $\Gamma(G,S)$ is at most
  \begin{equation}\label{eq:udun}(\log |G|)^C,\end{equation}
  where $C$ is an absolute constant.
\end{cor}
\begin{proof}
  Apply Theorem \ref{thm:asnat} to $A = S$, $A = S^3$, $A= S^9$, etc.
\end{proof}

The product theorem has other applications, notably to spectral gaps and expander graphs
(\cite{MR2415383}, \cite{MR2587341}, \cite{MR2892611}). It has been generalized several
times, to the point where now it is known to hold for all finite simple\footnote{It is trivial to see that the theorem and
corollary above still hold if $\SL_2(\mathbb{Z}/p\mathbb{Z})$ is replaced by the simple
group $\PSL_2(\mathbb{Z}/p\mathbb{Z})$.} groups of Lie type
and bounded rank \cite{BGT}, \cite{MR3402696}. When we say {\em bounded rank}, we mean that
the constants $\delta$ and $C$ in these generalizations of Thm.~\ref{thm:asnat} and
Cor.~\ref{cor:nonat} depend on the rank of the group $G$. 

Babai's conjecture states that the bound (\ref{eq:udun}) holds any finite, simple,
non-abelian $G$, and any set of generators $S$ of $G$, with $C$ an {\em absolute}
constant (i.e., one constant valid for all $G$). By the classification of finite
simple groups (henceforth: CFSG),
every finite, simple, non-abelian group $G$ is either (a) a simple group
of Lie type, or (b) an alternating group $\Alt(n)$, or (c) one of a finite number of
sporadic groups. Being finite in number, the sporadic groups are irrelevant for the
purposes of the asymptotic bound (\ref{eq:udun}). It remains, then, to consider
whether Babai's conjecture is true for $\Alt(n)$, and for simple, finite
groups of Lie type whose rank goes to infinity.

Part of the problem is that, in either of these two cases, the natural generalization of
Thm.~\ref{thm:asnat} is false: counterexamples due to Pyber and Spiga
\cite{MR2898694}, \cite{MR2876252} show that
$\delta$ has to depend on the rank of $G$, or on the
index $n$ in $\Alt(n)$, at least if there are no additional conditions.
Nevertheless, Babai's conjecture is still believed to be true.

Some of the ideas leading to Thm.~\ref{thm:asnat} and its generalizations were
useful in the proof of the following result, even though the overall argument looked
rather different.
\begin{thm}[\cite{MR3152942}]\label{thm:pais}
  Let $G = \Alt(n)$ or $G=\Sym(n)$. Then,
   \begin{equation}\label{eq:pedestr}\diam G \leq e^{C (\log \log |G|)^4 \log \log \log |G|},\end{equation}
  where $C$ is an absolute constant.
\end{thm}
Here we write $\diam G$ for the ``worst-case diameter''
\[\diam G = \max_{A\subset G: G = \langle A\rangle} \diam(\Gamma(G,A)),
\]
i.e., the same sort of quantity that we bounded in
Cor.~\ref{cor:nonat}.
%\ref{thm:pais} bound.

Theorem \ref{thm:pais}
is not as strong as Babai's conjecture for $\Alt(n)$ or $\Sym(n)$,
since the quantity on the right of
(\ref{eq:pedestr}) is larger than $(\log |G|)^C$.
The proof of Thm.~\ref{thm:pais} did not
go through the proof of an analogue of a product theorem; it used another kind
of inductive process.

One of our main aims
in what follows is to give a different proof of Theorem \ref{thm:pais}.
Some of its elements are essentially the same as in the original proof, sometimes
in improved or simplified versions. Others are more closely inspired by the tools
developed for the case of groups of Lie type. 

Theorem \ref{thm:pais} -- or rather a marginally weaker version thereof
(Thm.~\ref{thm:molop}), with
an additional factor of $\log \log \log |G|$ in the exponent -- will follow as a direct consequence of the following
product theorem, which is new. It is, naturally, weaker than a literal analogue
of Thm.~\ref{thm:asnat}, since such as analogue would be false, by the counterexamples
we mentioned.
\begin{thm}\label{thm:jukuju}
Let $A\subset \Sym(n)$ be such that $A=A^{-1}$, $e \in A$, and
$\langle A\rangle$ is $3$-transitive. There are absolute constants
$C,c>0$ such that the following holds.
Assume that $|A|\geq n^{C (\log n)^2}$.
%(Alternatively, assume $A = S^k$ for
%some $S\subset \Sym(n)$, $k\geq n^{C \log n}$.)
Then either
\begin{equation}\label{eq:uru}
  |A^{n^C}|\geq |A|^{1+c \frac{\log \log |A|
    }{(\log n)^2 \log \log n}}\end{equation}
or
\begin{equation}\label{eq:rororo}
  \diam(\Gamma(\langle A\rangle,A))\leq n^C 
  \diam(G),\end{equation}
where $G$ is a transitive group on $m\leq n$ elements
such that either (a) $m\leq e^{-1/10} n$ or (b) $G\nsim \Alt(m), \Sym(m)$.
\end{thm}

Our general objective
will be to make the proof for $\Alt(n)$ and $\Sym(n)$ not just simpler
but closer to that for groups of Lie type of bounded rank.
Part of the motivation is that the next natural aim is to
study in depth groups of Lie type of unbounded rank, which combine features of both
kinds of groups.

{\em Overall idea.} To prove the growth of sets $A$ in a group $G$, we study the
actions of a group $G$.
First of all, every group acts on itself, by left and right multiplication,
and by conjugation. The study of these actions is always useful; it gives
us lemmas valid for every group. Then there are the actions 
that exist for a given kind of group.

A linear algebraic group acts by linear transformations on affine space. It
then makes sense to see how the action of the group affects varieties,
and what this tells us about sets of elements in the group.

In the case of the symmetric group $\Sym(\Omega)$, $|\Omega|=n$,
we have no such nicely geometric action.
What we do have is an action on a set $\Omega$, 
that, while completely unstructured,
is very small compared to the group. This fact allows us to use short random
walks to obtain elements whose action on $\Omega$ and low powers follows an
almost uniform distribution.

It is then unsurprising that the strategies for linear algebraic groups and
symmetric groups diverge: the actions that characterize the two kinds
differ. Nevertheless, it is possible to unify the strategies to some extent.
We shall see that the role played by {\em generic} elements -- in the sense
of algebraic geometry -- in the study of growth in linear algebraic groups
is roughly analogous to the role played in permutation groups by
{\em random} elements -- in the sense of being produced by random walks.

{\em Further perspectives.}
A ``purer'' product
theorem would state that either (\ref{eq:uru}) holds or, say,
$A^{n^{C \log n}} = G$.
The switch to diameters in conclusion (\ref{eq:rororo}) is not just
somewhat ungainly; it also slows down the recursion. If (\ref{eq:rororo})
were replaced by $A^{n^{C \log n}} = G$, we would
then obtain an exponent of $3$ instead of $4$ in Theorem \ref{thm:pais}.
Such a ``purer'' result is not contradicted by the existing counterexamples,
and so remains a plausible goal.

Yet another worthwhile goal would be to remove the dependence on the
Classification of Finite Simple Groups (CFSG). The proof here uses the
structure theorem in \cite{MR599634}/\cite{MR758332}, which relies on CFSG.
The proof in \cite{MR3152942} also depended on CFSG, for essentially
the same reason: it used \cite[Thm.~1.4]{zbMATH00091732}, which uses
\cite{MR599634}/\cite{MR758332}.

Incidentally, there is a flaw in \cite[Thm.~1.4]{zbMATH00091732} (proof
and statement), as L.~Pyber pointed out to the author. We fix it in
\S \ref{sec:babaiseress} (with input from Pyber);
the amended statement is in Prop.~\ref{prop:finbo}.
The bound in \cite{MR3152942} is not affected when we replace
\cite[Thm.~1.4]{zbMATH00091732} by Prop.~\ref{prop:finbo} in the proof of
\cite{MR3152942}.

{\em Notation.} 
We write actions on the right, i.e., if $G$ acts on $X$, and $g\in G$,
$x\in X$, we write $x^g$ for the element to which $g$ sends $x$.

As is usual,
we write $f(x) = O(g(x))$ to mean that there exists a constant $C>0$ (called an
{\em implied constant}) such that $|f(x)|\leq C g(x)$ for all large enough $x$. 
We also write $f(x)\ll g(x)$ to mean that $f(x) = O(g(x))$, and
$f(x)\gg g(x)$ to mean, for $g$ taking positive values, that there is a
constant $c>0$ (called, again, an implied constant) such that $f(x) \geq c g(x)$
for all large enough $x$. When we write $O^*(c)$, $c$ a non-negative real,
we simply mean a quantity whose absolute value is at most $c$.

Given $h\in G$,
we write $C(h)$ for the centralizer $\{g\in G: g h = h g\}$ of $h$.
Given $H\leq G$, we write $C(H)$ for the centralizer
$\{g\in G: g h = h g\; \forall h\in H\}$ of $H$.

As should be clear by now, and as is standard, we write $\Alt(\Omega)$ for
the alternating group on a set $\Omega$, and $\Alt(m)$ for the abstract group
isomorphic to $\Alt(\Omega)$ for any set $\Omega$ with $n$ elements. We
define $\lbrack n\rbrack = \{1,2,\dotsc,n\}$.

{\bf Acknowledgements.}
The author is supported by ERC Consolidator grant
648329 (GRANT) and by funds from his Humboldt professorship.
He is deeply grateful to L\'aszl\'o Pyber for his extremely
valuable suggestions and feedback, and to Henry Bradford and Vladimir
Finkelshtein, for a very careful reading and many corrections.
%The ultimate goal would be a proof abstract enough, and close enough to the proof for
%groups of Lie type, for it be generalizable to Lie groups of unbounded rank.  We
%are only taking some steps towards that goal.

\section{Toolbox}

\subsection{Special sets}

In the proofs of growth for groups of Lie type, some of the main tools are
statements on intersections with varieties. A typical statement is of the following
kind.

\begin{lem}\label{lem:oshut}
  Let $G = \SL_2(K)$, $K$ a finite field. Let $A\subset G$ be a set of generators of $G$
  with $A = A^{-1}$. Let $V$ be a
  one-dimensional irreducible subvariety of $\SL_2$. Then, for
  every $\delta>0$, either $|A^3|\geq |A|^{1+\delta}$ holds, or
  the intersection of $A$ with $V$ has
  \[\ll |A|^{\frac{1}{\dim \SL_2} + O(\delta)} = |A|^{1/3 + O(\delta)}\]
    elements. The implied constants depend only on the degree of $V$.
\end{lem}

Special statements of this kind were proved and used in
\cite{Hel08} and \cite{HeSL3}, and have been central to
the main strategy since then. They were fully generalized in \cite{MR3402696}. As it happens,
Larsen and Pink, in the course of their work on finite subgroups of linear groups,
had proven results of the same kind -- for subgroups $H$, instead of sets $A$, but
for all simple linear groups $G$. Their procedure was adapted in
\cite{BGT} to give essentially the same general result as in \cite{MR3402696}.

(Incidentally, the main purpose of Larsen and Pink was to prove without CFSG a series
of statements that follows from CFSG. For this purpose, they developed tools that
were, in some sense,
both concrete and general. It was these features that let the tools be generalized later
to sets, as opposed to subgroups. This is not the only time that
preexistent work on doing without CFSG has proved fruitful in this context; we will see another instance
when we examine random walks and permutation groups.)
 
There is an obvious difficulty in adapting such work to the study of permutation groups:
in $\Sym(n)$, there seems to be no natural concept of a ``variety'', let alone of its
degree and its dimension.

The approach we will follow here is to strip to the proof of a statement 
such as Lemma \ref{lem:oshut} to its barest bones, so that the main idea becomes
a statement about an abstract group. We will later be able to see how to apply it
to obtain a useful result on permutation groups.

The proof of Lemma \ref{lem:oshut} goes as follows. First, we show that, for
generic $g_1, g_2\in \SL_2(\overline{K})$, the map
$\phi:V\times V \times V\to G$ given by
\[\phi(v_0,v_1,v_2) = v_0 \cdot g_1 V g_1^{-1} \cdot g_2 V g_2^{-1}\] 
is {\em almost-injective}, in the sense that the preimage of a generic point
of the (closure of) the image is zero-dimensional. ``A generic point'' here means
``a point outside a subvariety of positive codimension''.
Similarly,
``for $g_1$, $g_2$ generic'' means that
the pairs $(g_1,g_2)$ for which the map $\phi$ is {\em not} almost-injective
lie in a variety of positive codimension $W$ in $\SL_2\times \SL_2$.
Now, because $A$ generates $G$, a general statement on {\em escape of subvarieties} shows
that there exists a pair $(g_1,g_2)\in A^k \times A^k$ outside $W$, where $k$
is a constant depending only on the degree of $V$. (``Escape of subvarieties''
was an argument known before \cite{Hel08}. The statement in \cite[Prop.~3.2]{MR2129706}
is over $\mathbb{C}$, but the
argument of the proof there is valid over an arbitrary field; see, e.g.,
\cite[Prop.~4.1]{HeSL3}.)

Then we examine the image of $(A\cap V)\times (A\cap V) \times (A\cap V)$ under
$\phi$. If $\phi$ is injective, then the image has exactly the same size as the
domain, namely, $|A\cap V|^3$. In general, for $\phi$ almost-injective, the image
will have size $\gg |A\cap V|^3$. At the same time, the image is contained in
$A^{1 + 2 k + 1 + 2k + 2k + 1 + 2k} = A^{8 k + 3}$. Hence
\[|A\cap V| \leq \left|A^{8 k + 3}\right|^{1/3}.\]

Let us prove an extremely simple general statement that expresses
the main idea of the statement we have just sketched.
\begin{lem}
Let $G$ be a group. Let $A,B\subset G$ be finite.
Then 
\[|A B^{-1}| \geq \frac{|A| |B|}{\left|A A^{-1} \cap B B^{-1}\right|}
.\]
In particular, if $A A^{-1} \cap B B^{-1} = \{e\}$, then
\[|A B^{-1}| \geq |A| |B|.\]
\end{lem}
The condition $A A^{-1} \cap B B^{-1} = \{e\}$ is fulfilled if, for instance,
$A\subset H_1$, $B\subset H_2$, where $H_1$, $H_2$ are subgroups of $G$
with $H_1\cap H_2 =\{e\}$.
\begin{proof}
Consider the map $\phi:A\times B\to A B^{-1} \subset G$ defined by
\[(a,b) \mapsto a b^{-1}.\]
Clearly, as with any map from $A\times B$ to $G$,
\begin{equation}\label{eq:lolalola}
|\im(\phi)| \geq \frac{|A\times B|}{\max_{x\in G} |\phi^{-1}(x)|},\end{equation}
and of course $|A B^{-1}|\geq |\im(\phi)|$.

So, let us bound $\phi^{-1}(x)$. Say $\phi(a,b)=x=\phi(a',b')$. Then
\begin{equation}\label{eq:doso}a^{-1} a' = b (b')^{-1}.\end{equation}
In particular, given $a$, $b$ and $b (b')^{-1}$, we can reconstruct
$a'$ and $b'$. Moreover, again by (\ref{eq:doso}),
 $b (b')^{-1} $ lies in $A A^{-1} \cap B B^{-1}$. Letting
$(a,b)$ be fixed, and letting $(a',b')$ vary among all elements of $\phi^{-1}(x)$, we see that
\[\left|\phi^{-1}(x)\right| \leq |A A^{-1} \cap B B^{-1} |.\]
By (\ref{eq:lolalola}), we are done.
\end{proof}

We can apply the same idea to obtain growth assuming only that an intersection
of many sets is empty.
\begin{lem}\label{lem:tagore}
Let $G$ be a group. Let $A_0, A_1, \dots, A_k \subset G$ be finite.
Then there is at least one $0\leq j\leq k-1$ such that
\[\left|A_j A_{j+1}^{-1}\right| \geq \frac{\mathbf{A}^{\frac{k+1}{k}}}{
  \left|\bigcap_{j=0}^{k} A_j A_j^{-1}\right|^{1/k}},\]
where $\mathbf{A}$ is the geometric average $(\prod_{j=0}^k |A_j|)^{1/(k+1)}$.
In particular, if \begin{equation}\label{eq:notew}
  \bigcap_{j=0}^{k} A_j A_j^{-1} = \{e\},\end{equation}  then
\[\left|A_j A_{j+1}^{-1}\right| \geq \mathbf{A}^{\frac{k+1}{k}}.\]
\end{lem}
We will typically apply this lemma to sets $A_j$ that are conjugates of
each other, and so all of the same size $\mathbf{A}$. If $A\subset H$,
$A_j = g_j A g_j^{-1}$ and $\bigcap_{j=0}^k g_j H g^{-1} = \{e\}$,
then condition (\ref{eq:notew}) holds.
\begin{proof}
  Consider the map \[\phi:A_0\times A_1\times\dotsc \times A_k\to
  A_0 A_1^{-1} \times A_1 A_2^{-1} \times \dotsc \times A_{k-1} A_k^{-1}\]
  given by
  \[(a_0,a_1,\dotsc,a_k) \mapsto \left(a_0 a_1^{-1}, a_1 a_2^{-1},\dotsc,
  a_{k-1} a_k^{-1}\right).\]
  Clearly,
  \[\prod_{j=0}^{k-1} |A_j A_{j+1}^{-1}| = |\im(\phi)|
  \geq \frac{\prod_{j=0}^k |A_j|}{\max_{x\in G} |\phi^{-1}(x)|}.\]

  Say $\phi(a_0,a_1,\dotsc,a_j) = x = \phi(a_0',a_1',\dotsc,a_j')$. Then,
  since $a_j a_{j+1}^{-1} = a_j' \left(a_{j+1}'\right)^{-1}$ for all $0\leq j<k$,
  we see that $a_j^{-1} a_j' = a_{j+1}^{-1} a_{j+1}'$ for all $0\leq j < k$.
  Thus, $(a_0',a_1',\dotsc,a_j')$ is determined by $(a_0,a_1,\dotsc,a_j)$
  and the single element
  \[a_0^{-1} a_0' = a_1^{-1} a_1' = \dotsc = a_k^{-1} a_k',\]
  which lies in $\bigcap_{j=0}^{k} A_j A_j^{-1}$. We conclude that
  \[|\phi^{-1}(x)|\leq \left|\bigcap_{j=0}^{k} A_j A_j^{-1}\right|.\]
\end{proof}

We will later see how to obtain the weak orthogonality condition
\begin{equation}\label{eq:radapar}
  \bigcap_{j=0}^k g_j H g_j^{-1} = \{e\}\end{equation}
  for some kinds of subgroups of permutation groups.

%For instance, a much-used statement is
%that of {\em escape from subvarieties}. Let $A$ be a set of generators of a linear
%group $G(K)$ over a field $K$. Let $V$ be a subvariety of $G$. Then, even though $A$
%may be contained in $V$, there is a $k$, depending only on the dimension, number and
%degree of the components of $V$, such that $A^k$ does not all lie in $V$. (If you
%wish, elements of $A$ have ``escaped'' from $V$.) This is stated and proved over
%$\mathbb{C}$ in \cite{}, but the induction argument there works in general. 

%\section{A shorter proof of Helfgott-Seress}

%As we will see towards the end of the section, this Proposition, taken together
%with a corrected version of Babai-Seress, gives the same final bound
%as Helfgott-Seress: $\diam(\Gamma(G,A)) \leq \exp((\log |G|)^4 \log \log |G|)$
%for $G = \Sym(n), \Alt(n)$, or, in fact, for $G$ any transitive group on
%$n$ elements. 

\subsection{Subgroups and quotients}

We will need a couple of basic lemmas on subgroups and quotients. As explained
in \cite[\S 3.1--3.2]{MR3152942} and \cite[\S 4.1]{MR3348442}, they are
all easy applications of an orbit-stabilizer principle for sets
\cite[Lemma 4.1]{MR3348442}. We can also prove them by using the pigeonhole
principle directly. 

For $G$ a group and $H \le G$, we write $\pi_{G/H}: G \to G/H$ for the map
taking each $g \in G$ to the right coset $Hg$ containing $g$. Thus, for
instance, $|\pi_{G/H}(A)|$ equals the number of distinct cosets $H g$
intersecting $A$.

\begin{lem}[{\cite[Lem.\ 7.2]{HeSL3}}]\label{lem:duffy} 
Let $G$ be a group and $H$ a subgroup thereof. Let $A\subseteq G$ be a 
non-empty finite set. Then
\begin{equation}\label{eq:vento}
  |A A^{-1} \cap H| \geq \frac{|A|}{|\pi_{G/H}(A)|}
  \geq \frac{|A|}{\lbrack G:H\rbrack}.\end{equation}
\end{lem}
\begin{proof}
  By pigeonhole, there is a coset $H g$ of $H$
  containing at least $|A|/|\pi_{G/H}(A)|$
  elements of $A$. Fix $g_0\in A\cap H g$.
  Then, for each $g_1\in A\cap H g$, we obtain a distinct element
  $g_0 g_1^{-1} \in A A^{-1}\cap H$.
\end{proof}

\begin{lem}\label{lem:durdo} 
Let $G$ be a group and $H$ a subgroup thereof. Let $A\subseteq G$ be a 
non-empty finite set. Then, for any $k\geq 1$,
\begin{equation}\label{eq:verento}
  \left|A^{k + 1}\right|
  \geq \frac{\left|A^k \cap H\right|}{
   \left|A A^{-1} \cap H\right|} \cdot |A|.\end{equation}
\end{lem}
In other words, growth in a subgroup implies growth in the group.
\begin{proof}
  It is clear that
  \[\left|A^{k + 1}\right| \geq
  \left|\left(A^k \cap H\right) \cdot A\right|
  \geq \left|\left(A \cap H\right)^k\right|\cdot \left|\pi_{G/H}(A)\right|.\]
  At the same time, by Lemma \ref{lem:duffy},
  \[\left|\left(A A^{-1} \cap H\right)\right| \cdot \left|\pi_{G/H}(A)\right|
\geq |A|.\]
\end{proof}

\begin{lem}[{\cite[Lem.\ 3.7]{MR3152942}}]\label{lem:subcos}
Let $G$ be a group, let $H,K$ be subgroups of $G$ with $H\leq K$, and let
$A\subseteq G$ be a non-empty finite set. 
Then
\[|\pi_{K/H}(A A^{-1}\cap K)| \geq \frac{|\pi_{G/H}(A)|}{|\pi_{G/K}(A)|}
\geq \frac{|\pi_{G/H}(A)|}{\lbrack G:K\rbrack}.\]
\end{lem}
In other words: if $A$ intersects $r \lbrack G:H\rbrack$ 
cosets of $H$ in $G$, then $A A^{-1}$ intersects at least 
$r \lbrack G:H\rbrack/\lbrack G:K\rbrack = r \lbrack K:H\rbrack$ 
cosets of $H$ in $K$. (As usual, all of our cosets are right cosets $H g$,
$K g$, etc.)
We quote the proof in \cite[Lem.\ 3.7]{MR3152942}.
\begin{proof}
Since $A$ intersects 
$|\pi_{G/H}(A)|$ cosets of $H$ in $G$ and
$|\pi_{G/K}(A)|$ cosets of $K$ in $G$, and every coset of $K$ in $G$
is a disjoint union of cosets of $H$ in $G$, the pigeonhole principle implies
that there exists a coset $K g$ of $K$ such that $A$
intersects at least
$k = |\pi_{G/H}(A)|/|\pi_{G/K}(A)|$ cosets $H a \subseteq K g$. 
Let $a_1,\ldots,a_k$ be elements of $A$ in distinct 
cosets of $H$ in $Kg$. Then $a_i a_1^{-1}\in AA^{-1}\cap K$ for each
 $i=1,\ldots, k$. 
Note that $H a_1 a_1^{-1},\ldots,H a_k a_1^{-1}$ are $k$ distinct 
cosets of $H$.
\end{proof}

\subsection{Graphs and random walks}

For us, a graph is a directed graph, that is, a pair $(V,E)$, where
$V$ is a set and $E$ is a subset of the set of {\em ordered}
pairs of elements of $V$. (We allow loops, that is pairs, $(v,v)$.) 
A multigraph is the same as a graph, but with $E$ a multiset, i.e.,
edges may have multiplicity $>1$.

Given a group $G$ and a set of generators $A\subset G$, the {\em Cayley graph}
$\Gamma(G,A)$ is defined to be the pair $(G,\{(g,g a):g\in G, a\in A\})$.
It is connected because $A$ is a set of generators. Given a group $G$,
a set of generators $A\subset G$ and a set $X$ on which $G$ acts,
the {\em Schreier graph} $\Gamma(G,A;X)$ is the pair $(X,\{(x,x^a):
x\in X, a\in A\})$.

We take a {\em random walk} on a graph or multigraph $\Gamma$ by starting 
at a given vertex $v_0$ and deciding randomly, at each step, to which
neighbor $w$ of our current location $v$ to move. (A {\em neighbor} of $v$ is
a vertex $w$ such that $(v,w)\in E$.) We choose $w$ with
uniform probability among the neighbors of $v$, if $\Gamma$ is a graph,
or with probability proportional to the multiplicity of $w$, if $\Gamma$
is a multigraph.

In a {\em lazy random walk}, at each step, we first throw a fair coin to decide
whether we are going to bother to move at all. (Of course, if we decide to
move, and $(v,v)$ is an edge, we might move from $v$ to itself.) Our
random walks will always be lazy, for the sake of eliminating some
technicalities.

We say that the {\em $(\ell_\infty,\epsilon)$-mixing time} in a regular,
symmetric (multi)graph $\Gamma=(V,E)$ is at most $t$ if, for every (lazy)
random walk
of length $\geq t$, the probability
that it ends at any given vertex lies between $(1-\epsilon)/|V|$ and
$(1+\epsilon)/|V|$. We will use the fact that (multi)graphs with few vertices
have small mixing times.

\begin{prop}\label{prop:chudo}
Let $\Gamma$ be a connected, regular and symmetric multigraph of
valency $d$ and with $N$ vertices. Then the $(\ell_\infty,\epsilon)$-mixing 
time is at most $N^2 d \log(N/\epsilon)$.
\end{prop}
\begin{proof}
  This is a well-known fact; see, e.g., the exposition in
  \cite[\S 6]{MR3348442}. The main idea is to study the spectrum of the
       {\em adjacency operator}, meaning the
       operator $\mathscr{A}$ taking each function $f:V\to \mathbb{C}$ to
       a function $\mathscr{A}f$ whose value at $v$ is the average of $f(w)$
       over the neighbors $w$ of $v$ in the graph $\Gamma$. The connectedness
       of $\Gamma$ is used to show that, for every non-constant eigenfunction
       of $\mathscr{A}$, the corresponding eigenvalue $\lambda$ cannot
       be too close to $1$; it is at most $1-1/N^2 d$. The bound on the mixing
       time then follows.
\end{proof}

In particular, Prop.~\ref{prop:chudo} holds when $\Gamma$ is any Schreier graph
$\Gamma(G,A;\Omega^{(k)})$ of the action of a permutation group $G\leqslant \Sym(\Omega)$ on the
set $\Omega^{(k)}$ of $k$-tuples of distinct elements of $\Omega$.
The point is that $N = \left|\Omega^{(k)}\right|\leq |\Omega|^k/k!$ is very
small compared to $\Sym(\Omega)$ (which is of course of size $|\Omega|!$)
for $k$ bounded. We can make sure that $A$ is small as well, by 
the following simple lemma.

\begin{lem}\label{lem:dustu}
  Let $A\subset \Sym(n)$. Then there is a subset $A_0\subset A\cup A^{-1}$
  such that $\langle A_0\rangle = \langle A\rangle$, $|A_0|\leq 4 n$
  and $A_0 = A_0^{-1}$.
\end{lem}
\begin{proof}
Choose an element 
$g_1\in A$, and then an element $g_2\in A$ such that $\langle g_1\rangle
\lneq \langle g_1,g_2\rangle$, and then an element $g_3\in A$ such that
$\langle g_1,g_2\rangle \lneq \langle g_1,g_2,g_3\rangle$, \dots
Since the longest subgroup chain in $\Sym(n)$ is of length $\leq 2n-3$ 
\cite{MR860123}, we must stop in $r\leq 2n-3<2n$ steps. Let
$A_0 = \{g_1,g_1^{-1},\dotsc,g_r,g_r^{-1}\}$.
\end{proof}

The point is that, while we cannot assume we can produce random,
uniformly distributed elements of $G=\langle A\rangle$
as short products in $A$ (we cannot assume what we are trying to prove,
namely, that the diameter is small), we can take short, random products
of elements of $A$, and their action on $\Omega^{(k)}$ is like that
of random, uniformly distributed elements. This observation was already used
in \cite{BBS04} to prove the following.
\begin{lem}[\cite{BBS04}]\label{lem:bbs}
  Let $A\subset \Sym(\Omega)$, $|\Omega|=n$,
  be such that $A=A^{-1}$ and $G=\langle A\rangle$
  is $3$-transitive. Assume there is a $g\in A$, $g\ne e$, with
  $|\supp(g)|\leq (1/3-\epsilon) n$, $\epsilon>0$. Then
  \[\diam(\Gamma(G,A)) \ll_\epsilon n^8 (\log n)^c,\]
  where $c$ is an absolute constant.
\end{lem}
\begin{proof}[Sketch of proof]
  See \cite{BBS04} or the exposition in \cite[\S 6.2]{MR3348442}.
  The main idea is as follows. Let $A_0$ be as in Lemma \ref{lem:dustu},
  and let $h\in A_0^{m} \subset A^{m}$ be the outcome
  of a random walk on $A_0$ of length $\leq m$, where
  $m \geq 4 n^3 \log(n/\epsilon')$ and $\epsilon'=\epsilon/100$ (say). Then,
  by Prop.~\ref{prop:chudo}, for any $x,y\in \Omega$, the probability
  that $h$ takes $x$ to $y$ is almost exactly $1/n$. In particular,
  for $x\in \supp(g)$, the probability that $x^h\in \supp(g)$ is almost
  exactly $|\supp(g)|/n$. Hence
  \[\left|\supp(g) \cap \supp\left(h g h^{-1}\right)\right|\lesssim
  \frac{|\supp(g)|^2}{n} \leq \left(\frac{1}{3} - \epsilon\right)|\supp(g)|.\]
  A quick calculation shows that the commutator
  $\lbrack g, h^{-1}\rbrack = g^{-1} h g h^{-1}$ obeys
  \[|\supp\left(\left\lbrack g, h^{-1}\right\rbrack\right)| \leq 3
  \left|\supp(g) \cap \supp\left(h g h^{-1}\right)\right|,\]
  and so, for $g' = \left\lbrack g, h^{-1}\right\rbrack$,
  $|\supp(g')|\leq |\supp(g)|^2/n \leq (1-3\epsilon) |\supp(g)|$.

  We iterate until, after $O(\log \log n)$, we obtain an element $h$ of
  support of size $2$ or $3$. (Additional care is taken in the process so that
  our element $h$ is never trivial. It is, in fact, convenient to take
  $m \geq 4 n^5 \log(n/\epsilon')$ from the beginning, so that the probability
  that $h$ takes a pair $(x,x')\in \Omega^{(2)}$ to a pair
  $(y,y')\in \Omega^{(2)}$ is almost exactly $1/|\Omega|^{(2)} = 2/n(n-1)$.)
  We conjugate $h$ by elements of $A_0^{n^3}$
  to obtain a set $C$ consisting of all $2$-cycles or $3$-cycles. It is
  clear that $\diam(\Gamma(G,C))\ll n$.
\end{proof}

We shall now use short random walks to construct elements $g_j$ such that
a weak orthogonality condition in the sense of (\ref{eq:radapar}) holds
for some kinds of sets $B$. By an {\em orbit} of a set $B\subset \Sym(\Omega)$
we mean a subset of $\Omega$ of the form $x^B$, $x\in \Omega$.

\begin{lem}\label{lem:shortorb}
  Let $A,B\subset \Sym(\Omega)$, $|\Omega|=n$, $e\in B$. Let
  $0<\rho<1$.
  Assume that
  $\langle A\rangle$ is $2$-transitive, and that $B$ has no
  orbits of length $>\rho n$. Then there are
  $g_1, g_2,\dotsc,g_k\in \left(A\cup A^{-1}\cup \{e\}\right)^m$,
  $k\ll (\log n)/|\log \rho|$,
  $m\ll n^6 \log n$, such that
  \[\bigcap_{j=1}^k g_j B g_j^{-1} = \{e\}.\]
\end{lem}
If we required only that $g\in \langle A\rangle$ (and not $g_k\in
(A\cup A^{-1}\cup \{e\})^m$), and
$B$ were assumed to be a group,
then this Lemma would be the ``splitting lemma'' in \cite[\S 3]{Bab82}.
The fact that the proof can be adapted illustrates what we were saying:
short random products act on $\Omega^{(2)}$ as random elements do.
Prop.~5.2 in \cite{MR3152942} is an earlier generalization of Babai's
``splitting Lemma'', based on the same idea.
\begin{proof}
  Let $g_1,\dotsc,g_k$ be the outcome of $k$ independent
  random walks of length $\leq m$ on $A_0$, where
  $m \geq 4 n^6 \log(n/\epsilon)$, $\epsilon>0$ and
  $A_0\subset A \cup A^{-1}$ is as in Lemma \ref{lem:dustu}.
Then,
by Prop.~\ref{prop:chudo}, for any $(x,y), (x',y')\in \Omega^{(2)}$ and any
$1\leq j\leq k$,
the probability that $g_j$ takes $(x,y)$ to $(x',y')$ lies between
$(1-\epsilon)/\left|\Omega^{(2)}\right|$ and
$(1+\epsilon)/\left|\Omega^{(2)}\right|$.

Since $B$ has no orbits of length $> \rho n$,
there are at most $\rho \left|\Omega^{(2)}\right|$ pairs $(x',y')\in
\Omega^{(2)}$ such that $x'$ and $y'$ lie in the same orbit of $H$.
Hence, for any $(x,y)\in \Omega^{(2)}$ and any
$1\leq j\leq k$, the probability that $x^{g_j}$ and $y^{g_j}$
lie in the
same orbit of $B$ is at most $(1+\epsilon) \rho$. Since 
$g_1,\dotsc,g_k$ were chosen independently, it follows that the probability
that $x^{g_j}$ and $y^{g_j}$ are in the same orbit for every
$1\leq j\leq k$ is at most
$((1+\epsilon) \rho)^k$. 

Now,  $x^{g_j}$ and $y^{g_j}$ are in the same orbit for every
$1\leq j\leq k$ if and only if $x$ and $y$ are in the same orbit
of $B' = \bigcap_{j=1}^k g_j B g_j^{-1}$.
The probability that at least two distinct $x$, $y$ lie
in the same orbit of $B'$ is therefore at most
\[n^2 ((1+\epsilon) \rho)^m .\]
We let\footnote{We can assume $\rho\leq (n-1)/n$. Thus
  $\epsilon= \rho^{-1/2}-1$ implies $\epsilon\gg 1/n$, and so
  $\log(n/\epsilon) \ll \log n$.}
$\epsilon= \rho^{-1/2}-1$, so that $((1+\epsilon) \rho) = \rho^{1/2}$.
Then, for $k > 2 (\log n^2)/|\log \rho|$,
\[n^2 ((1+\epsilon) \rho)^m < 1.\]
In other words, with positive probability, no two distinct $x$, $y$ lie
in the same orbit of $B'$, i.e., $B'$ equals $\{e\}$.
Thus, there exist $g_1,\dotsc,g_m$ such that $B' = \{e\}$.
\end{proof}

It is a familiar procedure in combinatorics (sometimes
called the {\em probabilistic method}) to prove that a lion can be found
at a random place of the city with positive probability, and to conclude
that there must be a lion in the city. What we have done is prove that, after
a short random walk, we come across a lion with positive probability
(and so there is a lion in the city).

\begin{cor}\label{cor:lalmo}
  Let $A,B\subset \Sym(\Omega)$, $|\Omega|=n$. Let
  $0<\rho<1$. Assume that
  $\langle A\rangle$ is $2$-transitive, and that $B B^{-1}$ has no
  orbits of length $>\rho n$. Then there is a
  $g\in \left(A\cup A^{-1}\cup \{e\}\right)^m$,
  $m\ll n^6 \log n$, such that
  \[\left|B B^{-1} g B B^{-1} g^{-1}\right| \geq |B|^{1 + \frac{|\log \rho|}{\log n}}.\]
\end{cor}
%This corollary is stronger and more general than \cite[Prop.~5.2]{MR3152942}.
\begin{proof}
  By Lemma \ref{lem:shortorb} applied to $B B^{-1}$ rather than $B$, there are
  $g_1, g_2,\dotsc,g_k\in \left(A\cup A^{-1}\right)^m$,
  $k\ll (\log n)/|\log \rho|$,
  $m\ll n^6 \log n$, such that
  \[\bigcap_{j=1}^k g_j B B^{-1} g_j^{-1} = \{e\}.\]
  Hence, by Lemma \ref{lem:tagore}, with $A_j = g_j B B^{-1} g_j^{-1}$
  (and $g_0=e$, say), there is a $0\leq j\leq k-1$ such that
  \[\left|A_j A_{j+1}^{-1}\right|\geq \left|B B^{-1}\right|^{1 +\frac{1}{k}}
  \left|B\right|^{1 +\frac{1}{k}}.\]
  Since
  $\left|A_j A_{j+1}^{-1}\right| = \left|g_j B B^{-1} g_j^{-1}
  g_{j+1} B B^{-1} g_{j+1}^{-1}\right|
  =\left|B B^{-1} g B B^{-1} g^{-1}\right|$ for $g =g_j^{-1} g_{j+1}$, we are done.
\end{proof}

\begin{cor}\label{cor:ratherbab}
  Let $A\subset \Sym(\Omega)$, $|\Omega|=n$, with $A=A^{-1}$ and $e\in A$.
  Let $0<\rho<1$. Assume that
  $\langle A\rangle$ is $2$-transitive. Let $\Sigma\subset \Omega$ be such that
  $(A^4)_{(\Sigma)}$ has no orbits of length $>\rho n$. Then either
  \begin{equation}\label{eq:firstopt}
    \left|\Sigma\right|\geq \frac{|\log \rho|}{3 (\log n)^2} \log |A|
  \end{equation}
  or
  \begin{equation}\label{eq:secondopt}
    \left|A^l\right|\geq |A|^{1 + \frac{|\log \rho|}{3 \log n}}
  \end{equation}
  for some $l\ll n^6 \log n$.
\end{cor}
Compare to \cite[Cor.~5.3]{MR3152942}.
\begin{proof}
  Let $B = (A^2)_{(\Sigma)} = (A A^{-1})_{(\Sigma)}$. Since
  $B B^{-1}\subset (A^4)_{\Sigma}$, we see that $B B^{-1}$
  has no orbits of length $>\rho n$. Apply Corollary \ref{cor:lalmo}.
  We obtain that
  \[\left|A^{l}\right|\geq |B|^{1 + \frac{|\log \rho|}{n}}\]
  for $l=4 m+2$, $m\ll n^6 \log n$.
  At the same time, by Lemma \ref{lem:duffy},
  \[|B| = \left|A A^{-1} \cap \Sym(\Omega)_{(\Sigma)}\right|
  \geq \frac{|A|}{\lbrack \Sym(\Omega):\Sym(\Omega)_{(\Sigma)}\rbrack}
  \geq \frac{|A|}{n^{|\Sigma|}}.\]
  
  Hence, either (\ref{eq:firstopt}) holds, or
  \[|B|> \frac{|A|}{n^{\frac{|\log \rho|}{2 (\log n)^2} \log |A|}} =
  |A|^{1 - \frac{|\log \rho|}{2 \log n}},\] and so
  \[\left|A^{l}\right|\geq |A|^{\left(1 - \frac{|\log \rho|}{2 \log n}\right)
    \left(1 + \frac{|\log \rho|}{\log n}\right)} = |A|^{
   1 + \frac{|\log \rho|}{2 \log n} \left(1 - \frac{1}{\log n}\right)}.\]
  We can assume $1-1/\log n \geq 2/3$, as otherwise (\ref{eq:secondopt}) holds
  trivially. 
\end{proof}

\subsection{Generating an element of large support}

We will need to produce an element of $\Sym(n)$ of
very large support (almost all of $\{1,2,\dotsc,n\}$). It is not difficult
to carry out this task using short random walks. 

\begin{lem}\label{lem:chachava}
  Let $g\in \Sym(n)$ have support $\geq \alpha n$, $\alpha>0$. Let
  $A\subset \Sym(n)$ 
  generate a $2$-transitive group. Assume $A=A^{-1}$, $e\in A$.
  Then, provided that $n$ is larger than a constant depending only on
  $\alpha$, there are $\gamma_i\in A^{n^6}$, $1\leq i\leq \ell$,
  where $\ell=O((\log n)/\alpha)$,
  such that the support of
  \[\gamma_1 g \gamma_1^{-1} \cdot \gamma_2 g \gamma_2^{-1} \dotsb
  \gamma_\ell g \gamma_\ell^{-1}
  \]
  has at least $n-1$ elements.
\end{lem}
\begin{proof}
  Let $h_1, h_2 \in \Sym(n)$, $m_i = |\supp(h_i)|$. By Prop.~\ref{prop:chudo},
  a random walk of length $r = \lceil 4 n^5 \log(n^2/\epsilon)\rceil$ gives
  us an element $\sigma$ of $A^r$ sending any given pair of distinct elements
  $x,y \in \{1,\dotsc,n\}$ to any given pair of distinct elements
  $x',y'\in \{1,\dotsc,n\}$ with probability $(1+O^*(\epsilon))/n(n-1)$.
%  Set $\epsilon=1/n$.

  An element $x\in \{1,\dotsc,n\}$ can fail to be in the support of
  $h_1 \sigma h_2 \sigma^{-1}$ only if (a) $x\notin \supp(h_1)$, $x\notin
  \supp \sigma h_2 \sigma^{-1}$, or (b) $x\in \supp(h_1)$ and $\sigma h_2 \sigma^{-1}$
  sends $x^{h_1}$ to $x$. For $x$ random,
  case (a) happens with probability at most
  $(1-m_1/n)\cdot (1+\epsilon) (1-m_2/n)$.
  In case (b), $\sigma$ must send $x^{h_1}$ to an element that is not
  fixed by $h_2$, and, moreover, it must send $x$ to $x^{h_1 \sigma h_2}$.
  Now, we know that, even given that $\sigma$ sends an element $x_0$
  (in this case,
  $x_0 = x^{h_1}$) to some specific element $y_0$,
  it will still send any $x\ne x_0$ to any $y\ne y_0$ with almost equal
  probability. 
  Hence,
  \[\begin{aligned}
  \Prob(x\notin \supp(h_1 \sigma h_2\sigma^{-1})) &\leq
  (1+\epsilon) \left(\left(1-\frac{m_1}{n}\right)
  \left(1 - \frac{m_2}{n}\right) + \frac{m_1}{n} \frac{m_2}{n} \frac{1}{n-1}\right)
  \end{aligned}\]
  We set $\epsilon=1/n$ and assume $m_1,m_2<n$. Then we have
  \[\Prob(x\notin \supp(h_1 \sigma h_2\sigma^{-1})) \leq
  (1+\epsilon) 
\left(1-\frac{m_1}{n}\right)
  \left(1 - \frac{m_2}{n}\right) + \frac{1}{n}.\]
  The expected value of $n-|\supp(h_1\sigma h_2 \sigma^{-1})|$ is thus
  at least $(1+\epsilon) (n-m_1) (1-m_2/n) + 1/n$. Hence there is
  a $\sigma\in A^r$ such that  $n-|\supp(h_1\sigma h_2 \sigma^{-1})|$ is at
  least that much.

  We apply this first with $h_1 = h_2 = g$, and obtain a $\sigma_1 = \sigma$
  as above; define $g_1 = g \sigma_1 g \sigma_1^{-1}$. Then we iterate: 
  we let $h_1 = g_1$, $h_2 = g$, and obtain a $\sigma_2 = \sigma$ such that
  $g_2 = g_1 \sigma_2 g \sigma_2^{-1}$ has large support; and so forth,
  with $h_1 = g_{i-1}$, $h_2 = g$ at the $i$th step.
  We obtain
  \[1 - \frac{\supp(g_i)}{n}
\geq (1+\epsilon) \left(1 - \frac{\supp(g)}{n}\right)
  \left(1-\frac{\supp(g_{i-1})}{n}\right) + \frac{1}{n},\]
  where $g_0 = g$,  and so, for $r = (1+\epsilon) (1 - \supp(g)/n)$
  (which is $<1$) and
  $k\geq 0$,
\[1 - \frac{\supp(g_k)}{n} \geq r^k \left(1 - \frac{\supp(g)}{n}\right)
+ \frac{1}{(1-r) n} \geq r^k (1-\alpha) + \frac{1}{(1-r) n}.
\]
We let $k = \lceil (\log n)/(\log 1/r)\rceil$ and obtain
\[\supp(g_k) \geq n - 1 - \frac{1}{(1-r)}.\]
For $n\geq 2/\alpha$, we have $r\leq (1+\alpha/2) (1-\alpha) < 1 - \alpha/2$, and so
$1/(1-r)\leq 2/\alpha$ and $k\ll (\log n)/\alpha$.

We can assume $\supp(g_k)<n$, as otherwise we are done.
Now apply the procedure at the beginning with $h_1=h_2=g_k$. 
We obtain $\Prob(x\notin \supp(g_k \sigma g_k \sigma^{-1})) < 2/n$,
provided that $n$ is larger than a constant depending only on $\alpha$.
Hence there is a $\sigma\in A^r$ such that
$\supp(g_k \sigma g_k \sigma^{-1})\geq n-1$. Since
\[g_k \sigma g_k \sigma^{-1} = g \cdot \sigma_1 g \sigma_1^{-1} \dotsb
\sigma_k g \sigma_k^{-1} \cdot \sigma g \sigma^{-1}\cdot (\sigma \sigma_1) g (\sigma \sigma_1)^{-1} \dotsb (\sigma \sigma_k) g (\sigma \sigma_k)^{-1},\]
we set $\ell = 2k+2$ and are done.
\end{proof}

It may be useful to compare Lemma \ref{lem:chachava} to analogous results on
{\em random subproducts} in the sense of \cite{zbMATH01116363}.
Such results make weaker assumptions (transitivity instead of double transitivity) and give weaker conclusions (support $\geq n/2$ instead of support
$\sim n$; see \cite[Lemma 2.3.1]{zbMATH01849958}, \cite[Lemma 4.3]{MR3152942}).

%Of course D/n is the expected value of the size of the orbit you are in,
%and T/n is the expected value of the size of the square.

%Works out nicely up to size (of what?) about square root.

%Is T > \sqrt{n} D/3? That's the question.

%n^(2/3) and then 1s? D about n^(-1/3), T  of size n^(1/3).

%Keep track of T? 

%(s^2-s) (D/n)^2 + s (T/n)

%(T-D) (D/n)^2 + D (T/n)

%Can control T so it is not more than D^2/n?
 
\subsection{Stabilizers and stabilizer chains}

Let $A\subset \Sym(\Omega)$, $|\Omega|=n$. Given a subset
$\Sigma = \{\alpha_1,\alpha_2,\dotsc,\alpha_k\} \subset \Omega$, we write
$A_{(\Sigma)}$ and $A_{\Sigma}$ for the {\em pointwise} and {\em setwise}
stabilizers, respectively:
\[A_{(\Sigma)} = A_{(\alpha_1,\dotsc,\alpha_k)} =
\{g\in A: \alpha_j^g = \alpha_j\;\;\; \forall 1\leq j\leq k\},\]
\[A_\Sigma = A_{\{\alpha_1,\dotsc,\alpha_k\}} = \{g\in A: \Sigma^g = \Sigma\}.\]

A {\em stabilizer chain} is simply a chain of
subsets
\[A \supset A_{(\alpha_1)} \supset A_{(\alpha_1,\alpha_2)} \supset \dotsc,\]
where $\alpha_1, \alpha_2, \dotsc \in \{1,2,\dotsc,n\}$.
Stabilizer chains
have been studied starting with Sims \cite{MR0257203} 
(in the case of $A$ equal to a subgroup $H$).
It is useful to find long chains of stabilizers
such that the orbits \[\alpha_j^{A_{(\alpha_1,\dotsc,\alpha_{j-1})}}\] are long.

Why do we want stabilizer chains with long orbits? Here is one reason.
\begin{lem}\label{lem:basilic}
  Let $A\subset \Sym(\Omega)$, $|\Omega|=n$. Let $\rho \in (0,1)$.
  Let $\Sigma = \{\alpha_1,\alpha_2,\dotsc,\alpha_k\}\subset \Omega$
  be such that, for every $1\leq j\leq k$,
  \begin{equation}\label{eq:muladrink}
   \left|\alpha_j^{A_{(\alpha_1,\dotsc,\alpha_{j-1})}}\right|\geq \rho n.\end{equation}
  Then $A^k$ intersects at least $(\rho n)^k$ right cosets of
  $\Sym(\Omega)_{(\Sigma)}$, and the restriction of the
  setwise stabiliser $(A^{-k} A^k)_\Sigma$ to $\Sigma$ is a subset of
  $\Sym(\Sigma)$ with at least $\rho^k k!$ elements.
\end{lem}
This result was shown
in the proof of \cite[Lemma 3]{Pyb93} for $A$ a subgroup, and
in \cite[Lemma 3.19]{MR3152942} for general $A$.
\begin{proof}
  First of all, notice that $A^k$ sends $(\alpha_1,\alpha_2,\dotsc,\alpha_k)$
  to at least $(\rho n)^k$ distinct $k$-tuples. This is shown as follows.
  Let $1\leq j\leq k$. Let $\Delta_j$ denote the orbit
  $\alpha_j^{A_{(\alpha_1,\dotsc,\alpha_{j-1})}}$.
  For each $\delta\in \Delta_j$, choose an element $g_\delta \in A_{(\alpha_1,
    \dotsc,\alpha_{j-1})}$ sending $\alpha_j$ to $\delta$. Let
  $S_i = \{g_\delta: \delta\in \Delta_i\}$. Clearly, $|S_i|=|\Delta_i|$ and
  $S_i\subset A$. Now let $(s_1,s_2,\dotsc,s_k)$,
  $(s_1',s_2',\dotsc,s_k')$ be two distinct elements of $S_1\times \dotsb
  \times S_k$. Then $s_k \dotsb s_2 s_1$ and $s_k'\dotsb s_2' s_1'$ send
  $(\alpha_1,\alpha_2,\dotsc,\alpha_k)$ to two different $k$-tuples: if
  $j$ is the least index such that $s_j \ne s_j'$, then
  \[\alpha_j^{s_k s_{k-1}\dotsb s_j} = \alpha_j^{s_j} \ne
  \alpha_j^{s_j'} = \alpha_j^{s_k' s_{k-1}' \dotsb s_j'},\]
  and so \[\alpha_j^{s_k s_{k-1} \dotsb s_j s_{j-1} \dotsb s_1}
 \ne \alpha_j^{s_k' s_{k-1}' \dotsb s_j' s_{j-1} \dotsb s_1}
 = \alpha_j^{s_k' s_{k-1}' \dotsb s_j' s_{j-1}' \dotsb s_1'}.\]
 Hence $(\alpha_1,\alpha_2,\dotsc,\alpha_k)$ is sent to at least
 $|S_1|\dotsb |S_k| \geq (\rho n)^k$ distinct tuples by the action of
 $S_k S_{k-1} \dotsb S_1 \subset A^k$.

 In other words, $A^k$ intersects at $\geq (\rho n)^k$ cosets
 $(\Sym(\Omega))_{(\Sigma)} g$ of $(\Sym(\Omega))_{(\Sigma)}$. By Lemma \ref{lem:subcos},
 \[\begin{aligned}
 \pi_{(\Sym(\Omega))_{\Sigma}/(\Sym(\Omega))_{(\Sigma)}}(A^k A^{-k}\cap (\Sym(\Omega))_{\Sigma}) &\geq
 \frac{\pi_{\Sym(\Omega)//(\Sym(\Omega))_{(\Sigma)}}(A^k)}{
   \lbrack \Sym(\Omega):(\Sym(\Omega))_{\Sigma}\rbrack} \\ &\geq \frac{(\rho n)^k}{
   n (n-1) \dotsc (n-k+1)/k!} \geq \rho^k k!.\end{aligned}\]
 Now, two elements of $(\Sym(\Omega))_{\Sigma}$ lie in different cosets
 of $(\Sym(\Omega))_{(\Sigma)}$ if and only their restrictions to $\Sigma$ are distinct. Since $A^k A^{-k} \cap (\Sym(\Omega))_{\Sigma} = (A^k A^{-k})_{\Sigma}$, we have
 shown that the restriction $(A^{-k} A^k)_\Sigma$ to $\Sigma$ is of size at least
 $\rho^k k!$.
 \end{proof}

%Schreier's Lemma, which we are about to cite, concerns
%subgroups in general. It is often applied to stabilizers.

%\begin{lem}[Schreier]\label{lem:schreier}
%Let $G$ be a group and $H$ a subgroup thereof.
%Suppose $A$ intersects each coset of $H$ in $G$. Then
%$A A^{-2} \cap H$ generates $\langle A\rangle \cap H$.
%\end{lem}
%\begin{proof}
%  Standard. See, e.g., .
%  The same proof is given in \cite[Lemma 3.8]{MR3152942}
%  with the same notation we use here.
%\end{proof}

%\begin{cor}\label{cor:schreicor}
%  Let $A\subset \Sym(\Omega)$, $|\Omega|=n$. Assume $\langle A\rangle$ is transitive. Let $\alpha\in \Omega$. Then
%  \[\langle \left((A\cup A^{-1}\cup \{e\})^{3 n-3}\right)_\alpha\rangle =
%  \langle A\rangle_\alpha.\]
%\end{cor}
%\begin{proof}
%  Write $S_j = (A\cup A^{-1} \cup \{e\})^j$ for $j\geq 1$, $S_0=\{e\}$.
%  If $\alpha^{S_{j+1}} = \alpha^{S_j}$ for some $j$, then
%  $\alpha^{S_{j+k}} = \alpha^{S_{j+1}\cdot S_{k-1}} = \alpha^{S_j\cot S_{k-1}}
%  = \alpha^{S_{j+k-1}}$ for all $k\geq 1$, and so $\alpha^{S_{j'}}=\alpha^{S_j}$
%  for every $j'>j$. We would obtain a contradiction to the transitivity of
%  $\langle A\rangle$, unless $\alpha^{S_{j}} = \Omega$.
%  It follows that $\left|\alpha^{S_j}\right|$ increases
%  strictly with $j$ until it
%  reaches $n$, and so $\alpha^{S_{n-1}} = \Omega$.

%  In other words, $S_{n-1}$ intersects every coset of $\Sym(\Omega)_\alpha$.
%  Hence, \[\left(A\cup A^{-1}\cup \{e\})^{3 n-3}\right)_\alpha =
%  S_{n-1}^3 \cap \Sym(\Omega)_\alpha\]
%  generates
%  $\langle S_{n-1}\rangle \cap \Sym(\Omega)_\alpha = \langle A\rangle_\alpha$.
%\end{proof}

\subsection{Composition factors. Primitive groups}
%Lastly, we need the main result in \cite{zbMATH00091732}.
Let us recall some standard definitions.
A {\em composition factor} of a group $G$ is a quotient
$H_{i+1}/H_i$ in a composition series of $G$, i.e., a series
\[1=H_0\triangleleft H_1 \triangleleft \dotsc \triangleleft H_n =G,\]
where every quotient is simple. By the Jordan-H\"older theorem, whether
or not an abstract group is a composition factor of $G$ does not depend
on the particular composition series of $G$ being used.
%{\em I should probably omit the comment within parentheses, given the venue.}

A {\em section} of a group $G$
is a quotient
$H/N$, where $1\leq N\triangleleft H\leq G$. A composition factor is, by definition,
a section.

A {\em block system} for a permutation group $G\leq \Sym(\Omega)$ is a partition of $\Omega$
preserved by $G$, that is, a partition of $\Omega$
into blocks (sets) $B_1,\dotsc,B_k$ such that, if $x,y\in B_i$, $g\in G$ and $x^g\in B_j$,
then $y^g \in B_j$. A maximal block system is one that has blocks of size
$>1$ and cannot be subdivided into a finer
partition with blocks of size $>1$. (In other words,
it is a system of {\em minimal} non-trivial blocks.) A minimal block system is one that is not
the refinement of any block system other than the trivial partition of $\Omega$ into
one set $\Omega$.

The group $G$ is {\em primitive} if it has no block systems with more
than $1$ and fewer than $|\Omega|$ blocks, i.e., no block systems other than
(a) the partition of $\Omega$ into the single set $\Omega$ and
(b) the partition of $\Omega$ into one-element sets.
It follows from the definitions that
a group $G\leq \Sym(\Omega)$ acts as a primitive group on any minimal block system.

\subsection{Tools from the theory of permutation groups}
The following result guarantees the existence
of an element of small support in a group under rather mild conditions.
It is essentially due to Wielandt. Thanks are due to L. Pyber for the reference.
%Recall that the support $\supp(g)$ of an element $g\in \Sym(n)$ is simply the
%set of all $x\in \{1,\dotsc,n\}$ such that $x^g \ne x$.
\begin{lem}\label{lem:wielandt}
  For any $\epsilon>0$, there are $C_1, C_2\geq 0$ such that the following
  holds. Let $n\geq C_1$.
  Let $G\leqslant\Sym(\Omega)$, $|\Omega|=n$,
  be a group containing a section isomorphic to
  $\Alt(k)$ for $k\geq C_2 \log n$. Then there is a $g\in G$, $g\ne e$,
  such that $|\supp(g)|< \epsilon n$.
\end{lem}
This Lemma replaces \cite[Lem.~3.19]{MR3152942}, which was based on
\cite[Lem.~3]{MR894827}. 
\begin{proof}
  Suppose there is no $g\in G$, $g\ne e$, such that
  $|\supp(g)|<\epsilon n$; we say $G$ has {\em minimal degree}
  at least $\epsilon n$. We can assume without loss of generality that
  $C_2 \log n \leq k \leq 2 C_2 \log n$, since, given that $G$ contains
  a section isomorphic to $\Alt(k)$, it contains a section isomorphic to
  $\Alt(k')$ for all $k'\leq k$.
  
  Let $\omega = \min(\epsilon,0.4)$. Then, by
  \cite[Thm.~5.5A]{DM}, we have $n>\binom{k}{s}$ for $s = \lfloor
  \mu  (k+1)\rfloor$, $\mu = (1-\omega)^{1/5}$. By Stirling's formula,
  \[\binom{k}{s} = \frac{k!}{s! (k-s)!} 
  \gg_\omega \frac{1}{\sqrt{k}} \left(\frac{1}{\mu^\mu (1-\mu)^{(1-\mu)}}\right)^k\]
  for $k$ greater than a constant depending only on $\mu$. This gives
  a contradiction with $C_2 \log n\leq k\leq 2 C_2 \log n$ for $n\geq C_1$ when
  $C_1$ and $C_2$ are large enough in terms of $\epsilon$.
  (We use the condition $k\leq 2 C_2 \log n$ to ensure that the effect of
  $1/\sqrt{k}$ is negligible.)
\end{proof}

We also need a result telling us that a large subset of $\Sym(\Sigma)$
generates a large symmetric subgroup in a few steps.
\begin{lem}\label{lem:amusi}
  Let $H\leqslant \Sym(\Sigma)$, $|\Sigma|=k$.
  Let $\rho \in (1/2,1)$. If $|H|\geq \rho^k k!$ and $k$ is larger than
  a constant depending only on $\rho$, then there exists an orbit
  $\Delta\subset \Sigma$ of $H$ such that $|\Delta|\geq \rho |\Sigma|$
  and $H|_\Delta$ is $\Alt(\Delta)$ or $\Sym(\Delta)$.
\end{lem}
%The generalization to general $A$ in \cite{MR3152942} used
%the ideas in Bochert's proof and some elementary arguments.
\begin{proof}
  By \cite[Thm.~5.2B]{DM}, which is a somewhat strengthened version
  of \cite[Lem.~1.1]{MR703984}.
%  (The proof relies on a classical theorem of Bochert's \cite{Boch89}.)
  We simply need to check that
  \[\lbrack \Sym(\Sigma):H\rbrack < \min\left(\frac{1}{2} \binom{k}{\lfloor
    k/2\rfloor}, \binom{k}{m}\right)\]
  for $m = \lceil \rho k\rceil$. This inequality follows from Stirling's
  formula for $k$ larger than a constant depending on $\rho$.
\end{proof}

\subsection{Diameter comparisons: directed and undirected graphs}

%oh, also: diameters (A vs. A\cup A^{-1} (Babai, tripling)

We wish to derive a (version of) Theorem \ref{thm:pais}, which is
a bound on the diameter of a directed graph, from Theorem \ref{thm:jukuju},
which is a statement on sets $A$ satisfying $A = A^{-1}$. It would
be natural to expect such a statement to imply only a bound on the
diameter of an undirected graph. As it happens, the distinction between
directed and undirected graph matters little in this context, thanks to
the following result.

\begin{lem}[\cite{zbMATH05771615}, Thm.~1.4]\label{lem:soti}
  Let $G$ be a finite group and $A$ a set of generators of $G$. Then
  \[\diam \Gamma(G,A) \ll (\log |G|)^3\cdot
  \left(\diam \Gamma\left(G,A\cup A^{-1}\right)\right)^2,\]
  where the implied constant is absolute.
\end{lem}

It is thus enough to prove Theorem \ref{thm:pais} (and analogous statements) for
sets $S$ satisfying $S = S^{-1}$: simply replace $S$ by $S\cup S^{-1}$, and
use Lemma \ref{lem:soti}.

\section{Finding few generators for a transitive group}\label{subs:ellafi}

Let us be given a set $A\subset \Sym(\Omega)$. When is it the case
that there is a small subset $A'$ of $(A\cup A^{-1} \cup \{e\})^{n^C}$ (say)
that generates $\langle A\rangle$, or at least a transitive subgroup of
$\langle A\rangle$? Can we put further conditions on the elements of $A'$,
such as, for instance, that they all be conjugates of each other?

Our motivation for considering this question is the following.
We will find it necessary to do what amounts to
bounding the size of the intersection of a slowly growing set $A$ with the
centralizer of an element of large support.
It would stand to reason that there should be stronger bounds than for
the intersection of $A$ with a subgroup without long orbits:
being in the centralizer is a more restrictive condition.

%We will be able
%to prove stronger bounds, though it is unclear whether they are qualitatively
%optimal.

%Our aim is to show that, if $A$ grows slowly and generates a ``large enough'' subgroup of $\Sym(\Omega)$ (a $2$-transitive subgroup, say), and $h$
%is an element of large support,
%then $|A\cap C(h)|\ll |A|^{1-\delta}$, where $\delta>0$ is, if not constant, then a very slowly decreasing function of $|\Omega|$. Equivalently, given $B\subset C(h)$, and
%a set $A$ generating a $2$-transitive subgroup of $\Sym(\Omega)$, we may set
%it as our task
%to show that there are elements $g_j\in (A\cup A^{-1}\cup \{e\})^\ell$,
%$\ell$ moderate ($\ell = n^{O(1)}$, say), such that a product of a couple of
%sets of the form $g_j B g_j^{-1}$ is $\geq |B|^{1+\delta}$, $\delta>0$.

Given one of our main tools (Lemma \ref{lem:tagore}; see in particular the remark between its proof and its statement), we can quickly reduce
this task to the following one: given an element $h$ of large support and
a set of generators $A$ of $G$, find $g_1,\dotsc,g_k \in \left(A\cup A^{-1}
\cup \{e\}\right)^\ell$ such that
\begin{equation}\label{eq:urdu}
  C(h) \cap g_1 C(h) g_1^{-1} \cap \dotsc \cap g_k C(h) g_k^{-1}\end{equation}
is equal to $\{e\}$, or at least very small.

It is clearly enough for the group
\begin{equation}\label{eq:purelllady}
  \langle h, g_1 h g_1^{-1}, \dotsc, g_k h g_k^{-1}\rangle\end{equation}
to be transitive: the centralizer $H$ of a transitive subgroup of $\Sym(n)$
is {\em semiregular} (that is, no element of $H$ other than $e$ fixes any point),
and thus has $\leq n$ elements.

Let us, then, show that there are
$g_1,\dotsc, g_k\in \left(A \cup A^{-1} \cup \{e\}\right)^\ell$, $k$ and
$\ell$ small, such that the group (\ref{eq:purelllady}) is transitive.
We will be able to prove what we want
with $k \ll \log \log n$, assuming that $\langle A\rangle$ is $4$-transitive.

(As we will later discuss, reaching the bound
$k\ll \log n$ is substantially easier. An analogous, but not identical,
result with $k\sim \log n$ can be found in \cite[Lemma 5.13]{zbMATH01116363}.)

Our proof will be by iteration, with the iterative step being given
by the next proposition. We will be working with partitions into
orbits, but will prove the proposition for general partitions.
Recall that, given two partitions $P$, $Q$ of a set $\Omega$, the {\em join}
$P\vee Q$ is the finest partition
that is coarser (not necessarily strictly so) than both $P$ and $Q$.
The {\em trivial} partition of $\Omega$ is the partition $\{\Omega\}$.

Given a partition $P$ of $\Omega$ and an element $x\in \Omega$, we define
$S_P(x)$ to be the element of $P$ containing $x$. Let
$s_P(x) = \left|S_P(x)\right|$.

The {\em total variation distance} between two probability measures
$\mu_1$, $\mu_2$ on a finite set $X$ is defined to be
\[\delta(\mu_1,\mu_2) = \max_{S\subset X} |\mu_1(S) - \mu_2(S)|.\]

Thus, for instance, $\mu$ is at total variation distance at most $\epsilon$
from the uniform distribution if $\mu(S) = |S|/|X| + O^*(\epsilon)$ for
every $S\subset X$. Suppose this is the case. Then,
given any function $f:X\to \mathbb{R}$ with $0\leq f(x)\leq T$ for
every $x\in X$, we can easily estimate the expected value $\mathbb{E}_\mu f(x)$
of $f(x)$ with respect to $\mu$: clearly,
\[f(x) = \int_0^T 1_{L(f,t)}(x) dt
\;\;\;\;\;\text{(``layer-cake decomposition''),}\]
where $L(f,t) = \{x\in X: f(x)\geq t\}$, and so
\[
\mathbb{E}_\mu(f(x)) = \int_0^T \Prob_\mu(x\in L(f,t)) dt =
\int_0^T \left(\frac{L(f,t)}{|X|} dt + O^*(\epsilon)\right) dt.\]
Applying the same idea to the uniform distribution, without the error
term $O^*(\epsilon)$, we obtain that
\begin{equation}\label{eq:richpry}
  \mathbb{E}_\mu(f(x)) = \frac{1}{|X|} \sum_{x\in X} f(x) + O^*(\epsilon T).
  \end{equation}

\begin{lem}\label{lem:coeur}
  Let $P$ be a partition of a finite set $\Omega$ with $|\Omega|=n$.
Let $m\geq 2$. Denote by $\rho$ the proportion of elements $x$ of $\Omega$
such that $s_P(x)\geq m$.
  
  Let $g\in \Sym(\Omega)$ be taken at random with a distribution such
  that, for any element $\vec{v}$ of the set $\Omega^{(2)}$
  of ordered pairs of distinct elements of $\Omega$,
  the probability distribution of $\vec{v}^g$ is at total variation
  distance $\leq \epsilon$ from the uniform distribution on $\Omega^{(2)}$.
  Assume that the same is true for the probability distribution of
  $\vec{v}^{g^{-1}}$ as well.
%
%  and any
%  that
%  $\vec{v}^g \in S$ is
%  $|S|/\left|\Omega^{(2)}\right| + O^*(\epsilon)$, where 
%$\epsilon\geq 0$, and the probability that 
%  $\vec{v}^{g^{-1}} \in S$ is   $|S|/\left|\Omega^{(2)}\right| + O^*(\epsilon)$
%  as well.
  
  \begin{enumerate}
  \item\label{it:turandot1}
    With positive probability, the proportion of elements $x$ of $\Omega$
    such that $s_{P\vee P^g}(x)\geq m$ is $\geq 1- (1-\rho)^2 - \epsilon$.
  \item\label{it:turandot2}
    Assume that $\epsilon\leq \rho/100$, $n\geq 100$ and $2\leq m\leq n/2$.
    Then,
    with positive probability, the proportion of elements $x$ of $\Omega$
    such that $s_{P\vee P^g}(x)\geq (1+\rho/3) m$ is $\geq \rho^2/8$.
  \item\label{it:pureplon} Assume that $\epsilon\leq \min(\rho/25,\rho/4m)$
    and $n\geq 250$. 
  Then, with positive probability, $P\vee P^g$ contains at least
  one set of size at least
  \[\min\left(\frac{\rho}{10} n, \frac{\rho-\epsilon}{2} m^2\right).\]
\end{enumerate}
\end{lem}
The proof is straightforward, in that we will proceed by taking
expected values. We are giving constants simply for concreteness; they have not been optimized. Conclusion (\ref{it:turandot2}) is substantially weaker
than what we could obtain by means of more complicated variance-based
arguments such as those we will use in the proof of Lemma \ref{lem:coeur2}.
\begin{proof}
  \noindent  
  Let $B$ be the set of all $x\in \Omega$ such that $s_P(x)<m$.
  For each $x\in B$, the probability that $x^g \in B$ is $\leq |B|/n + \epsilon
  = 1 - \rho + \epsilon$. Hence, the expected value of the number of
  $x\in B$ such that $x^g\in B$ is $\leq (1+\epsilon-\rho) |B| =
  (1 + \epsilon- \rho) (1-\rho) n \leq ((1-\rho)^2 + \epsilon) n$.
  It obviously follows that the number of such $x$
  is $\leq ((1-\rho)^2 + \epsilon) n$ with positive probability.
In other words, conclusion (\ref{it:turandot1}) holds.

  For each 
$x\in \Omega$ such that $s_P(x)\geq m$, 
we choose a subset $Z(x)\subset S_P(x)$ of size $m$, 
in such a way that, for every set $S$ in $P$, every element of $S$
is contained in exactly $m$ sets $Z(x)$, $x\in S$.
(For instance, we may identify each element of $P$ having
$m'\geq m$ elements with the set $\mathbb{Z}/m' \mathbb{Z}$, and then let
$Z(x) = \{x,x+1,\dotsc,x+m-1\} \mo m'$ for every
$x\in \mathbb{Z}/m' \mathbb{Z}$.
 We can easily see that every
element of $\mathbb{Z}/m' \mathbb{Z}$ is then contained in exactly $m'$ sets
$Z(x)$.)
For every $x\in \Omega$ such that $s_P(x)<m$,
we let $Z(x)=\emptyset$. We write $z(x)$ for $|Z(x)|$; we see that
$z(x)$ can take only the values $m$ or $0$.

We see immediately that
\begin{equation}\label{eq:kupcop}\begin{aligned}
  \sum_{x,x'\in \Omega} \left|Z(x) \cap Z(x')\right| &=
  \sum_{x\in \Omega} \sum_{y\in Z(x)} |\{x'\in S_P(x): y\in Z(x')\}|\\
  &= \sum_{x\in \Omega} \sum_{y\in Z(x)} m = \rho m^2 n,\end{aligned}\end{equation}
a fact that will be useful later.

 We write $Z_g(x)$ for $Z\left(x^{g^{-1}}\right)^g$, and
  $z_g(x)$ for $|Z_g(x)|$. By definition,
  \[Z_g(x)\subset \left(S_P\left(x^{g^{-1}}\right)\right)^g = S_{P^g}(x).\]
  Clearly,
  \begin{equation}\label{eq:colosh}
    s_{P\vee P^g}(x) \geq z(x) + z_g(x) - \left|Z(x)\cap Z_g(x)\right|.
  \end{equation}
  For any $x$,
  \[\begin{aligned}
  \mathbb{E}(z_g(x)) &= \mathbb{E}\left(\left|Z\left(x^{g^{-1}}\right)\right|
  \right) = m\cdot \Prob\left(x^{g^{-1}}\not\in B\right)\\
  &= m\cdot (\rho + O^*(\epsilon)),
  \end{aligned}\]
  since $B$ is the set of elements $x$ of $\Omega$ such that
  $Z(x)=\emptyset$, and $|Z(x)|=m$ for all $x\in \Omega\setminus B$.
  
  Given $x\in \Omega\setminus B$ and a $y\in Z(x)$,
  we should estimate the probability that $y$ is an element
  of $Z(x)\cap Z_g(x)$, where $g$ is, as always, taken at random.
  Evidently, if $y=x$, then $y\notin Z_g(x)$ if $Z_g(x)$ is empty, and
  $y\in Z_g(x)$ otherwise.
  If $y\ne x$, we have $y\in Z_g(x)$ if and only if $g^{-1}$ sends
  $(x,y)$ to an element of
  \[S = \{(x',y'): x'\in \Omega\setminus B, y'\in Z(x), y'\ne x'\}.\]
  The number of elements of $S$ is $|\Omega\setminus B|\cdot (m-1)$. Hence
  \[\begin{aligned}
  \Prob(y\in Z_g(x)) \leq \frac{|\Omega\setminus B|\cdot
(m-1)}{n (n-1)} +\epsilon \leq \frac{\rho m}{n} + \epsilon.
  \end{aligned}\]
  Therefore,
  \[\begin{aligned}
  \mathbb{E}\left( \left|Z(x)\cap Z_g(x)\right|\right) &\leq
  \Prob(x\in Z_g(x)) + \mathop{\sum_{y\in Z(x)}}_{y\ne x}
  \Prob(y\in Z_g(x))\\
  &\leq \rho + \epsilon + \mathop{\sum_{y\in Z(x)}}_{y\ne x'}
  \left(\frac{\rho m}{n} + \epsilon\right) \leq
  \rho \left(1 + \frac{m^2}{n}\right) + \epsilon m,\end{aligned}\]
  and so
  \begin{equation}\label{eq:matater}\begin{aligned}
    \mathbb{E}&\left(\frac{1}{\rho n} \sum_{x\in \Omega\setminus B}
    \left(z_g(x) - \left|Z(x)\cap Z_g(x)\right|\right)\right)\\
    &\geq \frac{1}{|\Omega\setminus B|} \sum_{x\in \Omega\setminus B} \left(m\cdot (\rho - \epsilon)
    - \rho \left( 1 + \frac{m^2}{n}\right) - \epsilon m\right)\\
    &= \rho m
    - \rho \left( 1 + \frac{m^2}{n}\right) - 2 \epsilon m
    .\end{aligned}\end{equation}
  Thus, with positive probability,
\[\frac{1}{\rho n} \sum_{x\in \Omega\setminus B}
    \left(z_g(x) - \left|Z(x)\cap Z_g(x)\right|\right)\geq \rho m  
  - \rho \left( 1 + \frac{m^2}{n}\right) - 2\epsilon m.\]
  The contribution of all $x\in \Omega\setminus B$ such that
  $z_g(x) - \left|Z(x)\cap Z_g(x)\right| \leq \rho m/3$ is at most
  $\rho m/3$. Each one of the other $x\in \Omega\setminus B$ contributes
  at most $m/\rho n$. Hence, the number of all $x\in \Omega \setminus B$ such that
  $z_g(x) - |Z(x)\cap Z_g(x)|>\rho m/3$ is 
  \[\begin{aligned}
  &\geq \frac{\left(\frac{2}{3}  m -
    \left(1+\frac{m^2}{n}\right)\right) \rho - 2 \epsilon m}{m/\rho n}\\
  &\geq
  \left(\frac{2}{3} -
  \left(\frac{1}{m} + \frac{m}{n}\right)\right) \rho^2 n - 2 \epsilon \rho n
    \geq \left(\frac{1}{6} - \frac{2}{n}\right) \rho^2 n
  - 2 \epsilon \rho n
  \geq \frac{\rho^2}{8} n,
  \end{aligned}\]
  where we use the assumptions $2\leq m\leq n/2$, $\epsilon\leq \rho/100$,
  $n\geq 100$. By (\ref{eq:colosh}),
  we obtain that conclusion (\ref{it:turandot2}) holds.
  It remains to prove conclusion (\ref{it:pureplon}).

%  and so
%  \[\mathop{\sum_{x,x'\in \Omega}}_{x\ne x'} \left|Z(x) \cap Z(x')\right| =
%  \rho m (m-1) n.\]

Let $x\in \Omega$.  
For each $y\in S_P(x)$, every element of 
$Z_g(y)$ lies in $S_{P^g}(y)$ and hence in $S_{P\vee P^g}(x)$.
 Therefore, by inclusion-exclusion, for 
$U\subset S_P(x)$ arbitrary, 
  \begin{equation}\label{eq:disgu} \left|S_{P\vee P^g}(x)\right|\geq 
  \left|\bigcup_{y\in U} S_{P^g}(y)\right|
\geq \sum_{y\in U} z_g(y)
  - \mathop{\sum_{y,y'\in U}}_{y\ne y'}
\left|Z_g(y) \cap Z_{g}(y')\right|
.\end{equation}
   By our assumptions on the distribution of $g$,
  \begin{equation}\label{eq:autir}\begin{aligned}
      \mathbb{E}\left(\sum_{y\in U} z_g(y)  \right)
  &= \sum_{y\in U} \mathbb{E}\left(z_g(y)\right) =
  \sum_{y\in U} \mathbb{E}\left(z\left(y^{g^{-1}}\right)\right) \\
  &\geq \sum_{y\in U} (\rho - \epsilon) m
  = (\rho - \epsilon)  m |U| ,\end{aligned}\end{equation}
and, similarly,
\begin{equation}\label{eq:autir2}
  \mathbb{E}\left(\sum_{y\in U} z_g(y)  \right) \leq
(\rho + \epsilon) m |U|.
\end{equation}

We can apply (\ref{eq:richpry}) to $X = \Omega^{(2)}$ and
$f(x,x') = |Z(x)\cap Z(x')|$, with the probability distribution on $X$
given by $(y,y')^{g^{-1}}$, where $(y,y')$ is a given element of $\Omega^{(2)}$
and $g$ is taken randomly in the sense we have been using throughout.
Then, by (\ref{eq:kupcop}) and the fact that $0\leq f(x,x')\leq m$ for
all $(x,x')\in X$,
\[\mathbb{E}\left(\left|Z_g(y) \cap Z_{g}(y')\right|\right) =
\frac{\rho m^2 n - m n}{n (n-1)} + O^*(\epsilon m).\]
Hence
  \begin{equation}\label{eq:rempert}\begin{aligned}
\mathbb{E}\left(\mathop{\sum_{y,y'\in U}}_{y\ne y'}
\left|Z_g(y) \cap Z_{g}(y')\right|\right)
&=
|U| (|U|-1) \left(\frac{\rho m^2 n - m n}{n (n-1)} + O^*(\epsilon m)\right).\\
\\
&\leq 
\left(\frac{\rho m^2}{n} + \epsilon m\right) \cdot |U|^2 
.\end{aligned}\end{equation}

%  As we are about to see,
%  the expectation bounds above  suffice to show that 
%$P\vee P^g$ contains a very large set (of size $\gg n$).

  Therefore, by (\ref{eq:disgu}),
\[\mathbb{E}\left(s_{P\vee P^g}(x)\right)\geq 
(\rho-\epsilon) m |U| - \left(\frac{\rho m^2}{n} + \epsilon m\right)
|U|^2.
\]

In general, the maximum of an expression $a t - b t^2$, $a,b>0$,
is of course attained when $t$ equals $t_0 = a/2b$; moreover, since
$a (t_0 - \Delta) - b (t_0 - \Delta)^2 = a^2/4b - b \Delta^2$,
we see that $a \lfloor t_0\rfloor - b (\lfloor t_0\rfloor)^2 \geq
a^2/4b - b$. We let $a = (\rho-\epsilon) m$,
$b = \rho m^2/n + \epsilon m$.

%, and see that
%$t_0 = (1-\epsilon)/2 (m/n+\epsilon/\rho)$.

Suppose first that $t_0>m$. Then we simply choose
any $x\in \Omega$ with $|S_P(x)|\geq m$, and choose $U\subset S_P(x)$ with
$|U|=m$. Since $t_0=a/2b>m$, we see that $a m - b m^2 = a m (1 - b m/a)
> a m/2$, and so
\[\mathbb{E}\left(s_{P\vee P^g}(x)\right) > \frac{a m}{2} =
\frac{\rho -\epsilon}{2} m^2.
\]

Now suppose that $t_0\leq m$. Then 
%\begin{equation}\label{eq:mermelo}
%m\geq \sqrt{\frac{1-\epsilon}{2 (1+\epsilon)} n}.\end{equation}
%Then $t_0\leq m$, and so
there is an $x\in \Omega$ such that
$|S_P(x)|\geq t_0$. We choose $U\subset S_P(x)$ with 
$|U| = \lfloor t_0\rfloor$, and obtain that
\[
\mathbb{E}\left(s_{P\vee P^g(x)}\right) \geq \frac{a^2}{4 b} - b
.\]
Clearly $1/(r_1+r_2)\geq \min(1/2 r_1,1/2 r_2)$ for $a,b>0$. Hence
\[\begin{aligned}\frac{a^2}{4 b} &= \frac{a^2/4}{\frac{\rho m^2}{n} + \epsilon m}
\geq \frac{a^2}{8} \min\left(\frac{n}{\rho m^2},\frac{1}{\epsilon m}\right)\\
&\geq \frac{(1-\epsilon/\rho)^2}{8}
\min\left(\rho n, \frac{\rho^2}{\epsilon} m\right)\geq
\frac{(1-\epsilon/\rho)^2}{8}
\min\left(\rho n, 4 \rho m^2\right),
%\min\left(\frac{\rho}{10} n,
%\frac{2}{5} \rho m^2\right),
\end{aligned}\]
where we use the assumption $\epsilon\leq \rho/4 m$.
Again by $\epsilon \leq \rho/4m$, we see that
$t_0=a/2b\leq m$ implies that
\[\begin{aligned}
(1-\epsilon) \rho m = a\leq 2 b m = 2 \frac{\rho m^3}{n} + 2 \epsilon m^2
\leq 2 \frac{\rho m^3}{n} + \frac{\rho m}{2},\end{aligned}
\]
and so
$4 \rho m^2 \geq (1-2\epsilon) \rho n$.
Therefore
\[\begin{aligned}
\frac{a^2}{4 b} - b
&\geq \frac{(1-\epsilon/\rho)^2}{8} (1- 2\epsilon) \rho n
- \left(\frac{\rho m^2}{n} + \epsilon m\right)\\
&\geq \frac{(1-\epsilon/\rho)^2}{8} (1- 2\epsilon) \rho n - \frac{1-2\epsilon}{4}
\rho - \epsilon m.\end{aligned}\]
If $m\geq \rho n/10$, then $P$ contained sets of size $\geq \rho n/10$ to
begin with, and hence so does $P\vee P^g$. If $m< \rho n/10$, we obtain
that
\[\begin{aligned}
\mathbb{E}\left(s_{P\vee P^g(x)}\right) &\geq \frac{a^2}{4 b} - b
\geq
\left(\frac{(1-\epsilon/\rho)^2}{8} (1-2\epsilon) - \frac{\epsilon}{10}
-\frac{1 - 2\epsilon}{4 n}\right) \rho n\\
&\geq \left(\frac{(24/25)^2}{8} \cdot \frac{23}{25} - \frac{1}{250} 
- \frac{1}{600}\right) n \geq 
0.10031 \rho n > \frac{\rho n}{10}
\end{aligned}\]
by the assumptions $\epsilon\leq \rho/25\leq 1/25$ and $n\geq 150$.
Thus, conclusion (\ref{it:pureplon}) holds.
\end{proof}

Simply using Lemma \ref{lem:coeur} repeatedly, we could
give a proof of Prop.~\ref{prop:verodiro} with $k$ in the
order of $\log n$. Our crucial induction step, allowing us
$k \ll \log \log n$, will be provided by the following Lemma.
The proof will proceed by variance-based bounds. (In other words, we will
be using Chebyshev's inequality.)

\begin{lem}\label{lem:coeur2}
  Let $P$ be a partition of a finite set $\Omega$ with $|\Omega|=n$.
Let $m\geq 2$. Denote by $\rho$ the proportion of elements $x$ of $\Omega$
such that $s_P(x)\geq m$.
Let $g\in \Sym(\Omega)$ be taken at random with a distribution
as in Lemma \ref{lem:coeur}, with $\epsilon\leq \min(1/1000,\rho m/n)$.

Assume that $\rho\geq 999/1000$ and $1000\leq m\leq \sqrt{n}/100$. Then,
  with positive probability,
  \[s_{P\vee P^g}(x)\geq \frac{m^2}{2}\]
for more than $n/2$ elements $x$ of $\Omega$.
\end{lem}
We will use several estimates in the proof of part (\ref{it:pureplon}) of
Lemma \ref{lem:coeur}.
%, since the conditions are at least as strict here as they are there.
We shall
use the same notation as in that proof: $Z(x)$, $z(x)$, $Z_g(x)$ and $z_g(x)$
are the same as there.
\begin{proof}
%  We will now work
%towards conclusion (\ref{it:moreplon}), assuming $m\geq \sqrt{n}/50$.
%We will use variance bounds to show that, with
%positive probability, many elements of $\Omega$ lie in sets in
%$P\vee P^g$ of size $\gg m^2$.

%In fact, we can simply assume that
%\[m\leq \sqrt{n/2},\]
%since otherwise we can just choose a subset $U\subset S(x)$ with $|U|=m$ for
%some $m$, and obtain that
%\[\begin{aligned}\mathbb{E}\left(s_{P\vee P^g(x)}\right) &\geq
%(1-\epsilon) \rho m^2 - (1+\epsilon) \frac{\rho m^4}{n}\\ &\geq
%\frac{\rho n}{2} \left((1-\epsilon) - (1+\epsilon) \frac{m^2}{n}\right) \geq
%\frac{\rho n}{2} \left((1-\epsilon) - \frac{(1-\epsilon)^2}{2}\right)
%  > \frac{\rho n}{6}.\end{aligned}\]

Let $f_g(x) = \sum_{y\in Z(x)} z_g(y)$. By (\ref{eq:autir}) and (\ref{eq:autir2}),
\[(\rho - \epsilon) m \cdot z(x) \leq \mathbb{E}(f_g(x)) 
\leq (\rho +\epsilon) m \cdot z(x).\]
   %  (So far, we have just used the fact that $g$ acts on $\Omega$ and $\Omega^{(2)}$ in an almost uniform fashion.)
  Therefore, writing
  $E_F = \frac{1}{n} \sum_{x\in \Omega} F(x)$,
  we see that
  \begin{equation}\label{eq:ebound}
  (\rho-\epsilon) \rho m^2 \leq \mathbb{E}\left(E_{f_g}\right)
  \leq (\rho+\epsilon) \rho m^2,\end{equation}
  where we take $g$ at random, as always. Let
\begin{equation}\label{eq:mastuerzo}
  R_g =   \frac{1}{n} \sum_{x\in \Omega} 
  \mathop{\sum_{y,y'\in Z(x)}}_{y\ne y'}
\left|Z_g(y) \cap Z_{g}(y')\right|.
  \end{equation}  
  % - (1+\epsilon)^2 \frac{E_2(P)^2}{n}.\end{aligned}\
% where 
Then, by (\ref{eq:rempert}),
  \begin{equation}\label{eq:vengado}
    \mathbb{E}\left(R_g\right)
    \leq \rho \left(\frac{\rho m^2}{n} + \epsilon m\right) m^2
 = \frac{\rho^2 m^4}{n} + \epsilon \rho m^3 \leq \frac{2 \rho^2 m^4}{n},\end{equation}
where we use the assumption $\epsilon\leq \rho m/n$.
  
  Let us now bound the expected value of
  $\sum_{x\in \Omega} f_g(x)^2$. Clearly
  \begin{equation}\label{eq:nolola}\begin{aligned}
  f_g(x)^2 &= \sum_{y,y'\in Z(x)} z_g(y) z_g(y')
  \\ &= \sum_{y\in Z(x)} z_g(y)^2 +
  \mathop{\sum_{y,y'\in Z(x)}}_{y\ne y'} z_g(y) z_g(y').\end{aligned}
    \end{equation}
  Now, for $x$ such that $Z(x)$ is non-empty,
  \begin{equation}\label{eq:lelo1}
  \mathbb{E}\left(\sum_{y\in Z(x)} z_g(y)^2\right) \leq
  \sum_{y\in Z(x)} (\rho+\epsilon) m^2
  = (\rho+\epsilon) m^3\end{equation}
  and
  %, by Cauchy-Schwarz (or $\rho m^2 > \rho^2 m^2$),
  \begin{equation}\label{eq:lelo2}\begin{aligned}
  \mathbb{E}\left(\mathop{\sum_{y,y'\in Z(x)}}_{y\ne y'} z_g(y)
  z_g(y')\right)
  &\leq %\frac{(1+\epsilon)^2}{1-1/n} 
  \mathop{\sum_{y,y'\in Z(x)}}_{y\ne y'} (\rho^2 +\epsilon) m^2 \\ &=
  (\rho^2+\epsilon) m^2\cdot (m^2-m).
  \end{aligned}\end{equation}
  Therefore,
  \begin{equation}\label{eq:raplusplus}\begin{aligned}
    \mathbb{E}\left(\frac{1}{n} \sum_{x\in \Omega} f_g(x)^2\right) &\leq
    (\rho + \epsilon) \rho m^3 +
    (\rho^2+\epsilon) \rho (m^4 - m^3)\\
   &=
    (\rho-\rho^2) \rho m^3 + (\rho^2+\epsilon) \rho m^4 .
\end{aligned}\end{equation}
  We have just established a bound on the expectation of the variance:
    for
    \begin{equation}\label{eq:elpuma}
    V_f = \frac{1}{n} \sum_{x\in \Omega} f(x)^2
    - \left(\frac{1}{n} \sum_{x\in \Omega} f(x)\right)^2
    \end{equation}
we quickly see, by (\ref{eq:ebound}) and (\ref{eq:raplusplus}), that
  \begin{equation}\label{eq:kostreko}\begin{aligned}
    \mathbb{E}\left(V_{f_g}\right)&=
    \mathbb{E}\left( \frac{1}{n} \sum_{x\in \Omega} f_g(x)^2\right)
    - \mathbb{E}\left(\left(\frac{1}{n} \sum_{x\in \Omega} f_g(x)\right)^2\right)\\
    &\leq     (\rho-\rho^2) \rho m^3 + (\rho^2+\epsilon) \rho m^4 
    - \left((\rho+\epsilon) \rho m^2\right)^2
        \\ &\leq (1-\rho) \rho^3 m^4 +
    \epsilon_1 m^4,\end{aligned}\end{equation} where
  \begin{equation}\label{eq:eps1def}
      \epsilon_1 = \epsilon \rho + \frac{(1-\rho) \rho^2}{m}.
    \end{equation}
%\begin{equation}\begin{aligned}
%    \epsilon_1 &=
%\frac{(1+\epsilon) (\rho^2-\rho^3)}{m} + (1+\epsilon) \rho^3 -
%(1-\epsilon)^2 \rho^4 - (1-\rho) \rho^3\\
%&\leq \epsilon \rho^3 (1+2\rho) + \frac{(1+\epsilon) (1-\rho) \rho^2}{m} 
%.\end{aligned}\end{equation}
We may call $V_{f_g}$ the variance of $f_g$, just as we may call $E_{f_g}$
the expectation of $f_g$. 

  Now we should give a bound on the variance $\mathbb{V}\left(E_{f_g}\right)$ of 
$E_{f_g}$. Clearly
  \begin{equation}\begin{aligned}
      \mathbb{E}\left(E_{f_g}^2\right) &=
      \mathbb{E}\left(\left(\frac{1}{n} \sum_{x\in \Omega} f_g(x)
      \right)^2\right) =
      \frac{1}{n^2}\cdot  \mathbb{E}\left(
      \sum_{x,x'\in \Omega} f_g(x) f_g(x')\right).\end{aligned}\end{equation}
  By the definition of $f_g(x)$ and the fact that, for every $y\in \Omega$ with
  $Z(y)\ne \emptyset$, $y\in Z(x)$ for exactly $m$ values of $x\in \Omega$,
  \begin{equation}\label{eq:hausarzt}
    \begin{aligned} \sum_{x,x'\in \Omega} f_g(x) f_g(x') &=
\sum_{x,x'\in \Omega} \sum_{y\in Z(x)\cap Z(x')} z_g(y)^2 +
\sum_{x,x'\in \Omega} \mathop{\sum_{y\in Z(x), y'\in Z(x')}}_{y\ne y'} z_g(y)
z_g\left(y'\right)\\
&= \mathop{\sum_{y\in \Omega}}_{Z(y)\ne \emptyset} m^2 z_g(y)^2 +
\mathop{\mathop{\sum_{y,y'\in \Omega}}_{Z(y),Z(y')\ne \emptyset}}_{y\ne y'} m^2 z_g(y)
z_g(y').
%\\
%&= \mathop{\sum_{y\in \Omega}}_{Z(y)\ne \emptyset} m^2 z_g(y)^2 +
%\mathop{\sum_{y,y'\in \Omega}}_{y\ne y'} z(y) z\left(y'\right) z_g(y)
%z_g(y').
    \end{aligned}
  \end{equation}
Much as in (\ref{eq:lelo1}) and (\ref{eq:lelo2}),
\[\mathbb{E}\left(\mathop{\sum_{y\in \Omega}}_{Z(y)\ne \emptyset} m^2 z_g(y)^2 
\right) \leq
 \mathop{\sum_{y\in \Omega}}_{Z(y)\ne \emptyset} m^2\cdot (\rho + \epsilon) m^2
= (\rho + \epsilon) \rho m^4 n
\] and
\[\begin{aligned}
\mathop{\mathop{\sum_{y,y'\in \Omega}}_{Z(y),Z(y')\ne \emptyset}}_{y\ne y'}
m^2 z_g(y) z_g(y') &\leq
\mathop{\mathop{\sum_{y,y'\in \Omega}}_{Z(y),Z(y')\ne \emptyset}}_{y\ne y'}
m^2 \cdot (\rho^2 + \epsilon) m^2\\
&\leq (\rho^2+\epsilon) \rho^2 m^4 n^2.
\end{aligned}\]
Hence
\[ \mathbb{E}\left(E_{f_g}^2\right)\leq
\frac{(\rho+\epsilon) \rho}{n} m^4  +
(\rho^2+\epsilon) \rho^2 m^4, \]
and so, by (\ref{eq:ebound}),
%and $E_2(P)\leq (1+\delta) E_1(P)^2$,
\begin{equation}\label{eq:mestron}\mathbb{V}\left(E_{f_g}\right) \leq 
  \frac{(\rho +\epsilon) \rho}{n} m^4 +
  ((\rho^2 +\epsilon) - (\rho-\epsilon)^2) \rho^2 m^4
  \leq \epsilon_2 m^4
  ,\end{equation}
where
\begin{equation}\label{eq:menest}
  \epsilon_2 = (1+2\rho) \rho^2 \epsilon + \frac{(\rho + \epsilon) \rho}{n}.
  \end{equation}
Here, of course,
$\mathbb{V}\left(E_{f_g}\right) =
\mathbb{E}\left(E_{f_g}^2\right) -
\mathbb{E}\left(E_{f_g}\right)^2 =
\mathbb{E}\left(\left(E_{f_g}-\mathbb{E}\left(E_{f_g}\right)\right)^2\right)$.

%  Write $1_{\lbrack T,\infty)}$ for the indicator function of the interval
%    $\lbrack T,\infty)$. Let $\rho\in (0,1)$ be a quantity to be set later.
%  By (\ref{eq:ebound}),
%  \[\begin{aligned}\mathbb{E}\left(1_{\lbrack \rho E_1(P)^2,\infty)} E_f\right) &\geq
%  \mathbb{E}\left(-\rho E_1(P)^2 + E_f\right)\\
%  &\geq
%  (1-\rho-\epsilon) E_1(P)^2,\end{aligned}\] and so, by
%  Cauchy-Schwarz,
%  \begin{equation}\label{eq:dors}\mathbb{E}\left(
%    1_{\lbrack \rho^2 E_1(P)^4,\infty)}\left(E_f^2\right)\right)
%  \geq  (1-\rho-\epsilon)^2 E_1(P)^4.\end{equation}
%  Let \[c_\rho=\frac{(1+\epsilon)^2 - (1-\delta) +\epsilon'}{(1-\rho-\epsilon)^2}.\] 
  By (\ref{eq:ebound}), (\ref{eq:vengado}), (\ref{eq:kostreko})
 and (\ref{eq:mestron}) and Cauchy-Schwarz, we conclude that
  \[\begin{aligned}\mathbb{E}&\left(
  E_{f_g}^2 - c_1 V_{f_g} - c_2 \left(E_{f_g} - \mathbb{E}\left(E_{f_g}\right)
  \right)^2
 - c_3 n \cdot R_g\right)\\
 &\geq \mathbb{E}\left(E_{f_g}\right)^2 - c_1 \mathbb{E}\left(V_{f_g}\right)
 -c_2 \mathbb{V}\left(E_{f_g}\right) - c_3 n \cdot \mathbb{E}\left(R_g\right)
 \geq K m^4
  \end{aligned}\]
  for any $c_1,c_2,c_3>0$, where
\[K = 
(\rho - \epsilon)^2 \rho^2 -
\left((1-\rho) \rho^3 + \epsilon_1\right) c_1
- \epsilon_2 c_2 - 2 \rho^2 c_3.\]
  We will choose $c_1$, $c_2$, $c_3$ so that $K$ is positive.
  Then the probability that
    \begin{equation}\label{eq:durud}
           V_{f_g}\leq \frac{E_{f_g}^2}{c_1},
        \;\;\;\;\; E_{f_g}\geq \sqrt{c_2} \left|E_{f_g} - \mathbb{E}\left(E_{f_g}\right)\right|,
\;\;\;\;\;
n R_g\leq \frac{E_{f_g}^2}{c_3} 
        %&\frac{1}{n} \sum_{x\in \Omega} f(x)^2
%  - \left(\frac{1}{n} \sum_{x\in \Omega} f(x)\right)^2
%  \leq c_\rho \cdot
%  
  \end{equation}
  will be positive. What happens when (\ref{eq:durud}) is the case?

\begin{enumerate}
\item First of all, $E_{f_g}\geq \sqrt{c_2} |E_{f_g} - \mathbb{E}(E_{f_g})|$ implies
  \begin{equation}\label{eq:pinor}\frac{\sqrt{c_2}}{\sqrt{c_2}+1}
    \mathbb{E}\left(E_{f_g}\right)\leq E_{f_g} \leq
\frac{\sqrt{c_2}}{\sqrt{c_2}-1} \mathbb{E}\left(E_{f_g}\right).\end{equation}
\item
  By Chebyshev's inequality, if (\ref{eq:durud}) is the case, then
  for any $\tau>0$, the number of $x\in \Omega$ such
  that \[1-\tau \leq \frac{f_g(x)}{E_{f_g}} \leq 1+\tau\]
  does not hold is at most $(n V_{f_g}/E_{f_g}^2)/\tau^2 \leq n/c_1 \tau^2$.
\item  By (\ref{eq:mastuerzo}) and the last inequality in (\ref{eq:durud}), for any $\tau'>0$,
  the number of $x\in \Omega$ such that
  \[
\mathop{\sum_{y,y'\in Z(x)}}_{y\ne y'}
\left|Z_g(y) \cap Z_{g}(y')\right| \leq \tau' E_{f_g}
  \]
  does not hold is $\leq R_g n/\tau' E_{f_g} \leq E_{f_g}^2/c_3 \tau' E_{f_g} =
  E_{f_g}/c_3 \tau'$.
\end{enumerate}

Hence, for $\geq (1-1/c_1 \tau^2) n - E_{f_g}/c_3 \tau'$ values of 
$x\in \Omega$, by (\ref{eq:disgu}) and (\ref{eq:ebound}),
\begin{equation}\label{eq:lerso}\begin{aligned}
\left|S_{P\vee P^g}(x)\right|&\geq {f_g}(x) - \tau' E_{f_g} \geq (1 - \tau - \tau') E_{f_g} 
\geq \frac{\sqrt{c_2}}{\sqrt{c_2}+1} (1-\tau - \tau') \mathbb{E}\left(E_{f_g}\right)\\
&\geq \frac{\sqrt{c_2}}{\sqrt{c_2}+1} (1-\tau - \tau')
(\rho - \epsilon) \rho  m^2.
\end{aligned}
\end{equation}
Moreover, by (\ref{eq:ebound}) and (\ref{eq:pinor}),
\[\frac{E_{f_g}}{c_3 \tau'} \leq \frac{\sqrt{c_2}}{\sqrt{c_2}-1} \frac{\mathbb{E}(E_{f_g})}{c_3 \tau'} \leq \frac{\sqrt{c_2}}{\sqrt{c_2}-1} \frac{(\rho +\epsilon) \rho}{c_3 \tau'} m^2.\]
Thus, by the assumption $m\leq \sqrt{n}/100$, we obtain that
\[\left(1-\frac{1}{c_1 \tau^2}\right) n - \frac{E_{f_g}}{c_3 \tau'}
\geq \rho' n\]
for \begin{equation}\label{eq:mantecado}
  \rho' = 1 -\frac{1}{c_1 \tau^2} -
  \frac{\sqrt{c_2}}{\sqrt{c_2}-1} \frac{(\rho +\epsilon) \rho}{100^2 c_3 \tau'}.
  \end{equation}

It is time to choose the parameters $c_1$, $c_2$, $c_3$, $\tau$ and $\tau'$.
We let \begin{equation}\label{eq:progro}
  c_1=\frac{1}{4\delta_1},\;\;\;\;\;
  \delta_1 = \frac{(1-\rho) \rho^3 + \epsilon_1}{\rho^4},\;\;\;\;\;
  c_2 = \frac{1}{4 \epsilon_2},\;\;\;\;\;
  c_3 = \frac{\rho^2}{8},\end{equation}
  and $\tau = 1/4$, $\tau' = 1/12$. Then
$K\geq (\rho-\epsilon)^2-3/4 > 0$, by our assumptions on $\rho$ and $\epsilon$.

In fact, since we are assuming $\epsilon\leq 1/1000$,
$m\geq 1000$, $\rho\geq 999/1000$ and 
$n\geq (100 m)^2 \geq 10^9$,
\[\epsilon_1 \leq \epsilon + \frac{1-\rho}{m} \leq 0.001001,\] 
\[\epsilon_2 \leq 3 \epsilon + \frac{(1+\epsilon)}{n} \leq 0.0030001,\]
\[\delta_1 \leq \rho^{-1} - 1 + \frac{\epsilon_1}{\rho^4} \leq
\frac{1000}{999} - 1 + \frac{0.001001}{(999/1000)^4} \leq 0.002007\]
%\[= \frac{3}{1000} + \frac{1001/1000}{1000 \rho^2} \leq
%0.00401
by (\ref{eq:eps1def}), (\ref{eq:menest}), and (\ref{eq:progro}).
Hence, by (\ref{eq:lerso}),
\[\begin{aligned}
s_{P\vee P^g}(x) &= \left|S_{P\vee P^g}(x)\right|
\geq \frac{\sqrt{c_2}}{\sqrt{c_2}+1} (1 - \tau - \tau') (\rho - \epsilon) \rho m^2\\
&\geq \frac{1}{1 + \sqrt{4 \epsilon_2}}  \left(1 - \frac{1}{4} - \frac{1}{12}\right) \frac{998}{1000} \frac{999}{1000} m^2
%\\
%&\geq \frac{1}{1 + \sqrt{4\cdot 0.0030001}} \cdot 0.664 m^2 \geq
%0.598 m^2
> \frac{m}{2}\end{aligned}\]
for at least $\rho' n$ elements $x$ of $\Omega$.
Moreover, by (\ref{eq:mantecado}),
\[\rho' \geq 1 - \frac{4 \delta_1}{\tau^2} - \frac{1}{1 - \sqrt{4 \epsilon_2}}
\frac{1 + \epsilon}{100^2 \cdot \frac{\rho^2}{8} \cdot \frac{1}{12}}
> 0.86 > \frac{1}{2}.\]
\end{proof}

\begin{prop}\label{prop:verodiro}
  Let $P$ be a partition of a finite set $\Omega$ with $|\Omega|=n$.
  Assume that at least $\geq \rho n$ elements of $\Omega$, $\rho>0$, lie
  in sets in $P$ of size $>1$.
  Let $A\subset \Sym(\Omega)$ be a set of generators of a $4$-transitive subgroup of  $\Sym(n)$. Let $h\in \Sym(\Omega)$ have support of size
  $n-c$, where $0\leq c<n$.
  
  Then there are $g_1,\dotsc,g_k \in
  \left(A \cup A^{-1}\cup \{e\}\right)^v$, $k = 
  O(\log \log n) + O_{\rho,c}(1)$,
  $v = O(n^{10})$, such that the partition $Q_k$ defined
  by  \[Q_0 = P,\;\;\;\;\;\;\;\;\;\;\;\;
  Q_j = Q_{j-1} \vee Q_{j-1}^{g_j h g_j^{-1}}\;\;\;\;\;\text{for
    $1\leq j\leq k$}\]
  is the trivial partition of $\Omega$.
\end{prop}
An example given by W. Sawin\footnote{On MathOverflow, in comments to
  \cite{286057}.}
suggests that $4$-transitivity is a necessary assumption.
\begin{proof}
  Let $A_0\subset A\cup A^{-1}$ be as in Lemma \ref{lem:dustu}, and let
  $g\in A_0^v\subset \left(A\cup A^{-1}\right)^v$
  be the outcome of a random walk of length $v$,
  where $v = \lceil n^9 \log(n^4/\epsilon)\rceil$ for a given $\epsilon>0$.
  Then, by Prop.~\ref{prop:chudo} applied to the Schreier graph
  $\Gamma(G,A_0;\Omega^{(4)})$, given any two elements $\vec{v}_1$,
  $\vec{v}_2$ of the set $\Omega^{(4)}$ of quadruples of distinct
  elements of $\Omega$, the probability that $g$ takes $\vec{v}_1$
  to $\vec{v}_2$ lies between $(1-\epsilon)/\left|\Omega^{(4)}\right|$
  and $(1+\epsilon)/\left|\Omega^{(4)}\right|$.

  A moment's thought shows that, since the support of $h$ is of size
  $n-c$, the number of quadruples $(r,s,r',s')\in \Omega^{(4)}$ such
  that $r^h = r'$ and $s^h = s'$ is at least $(n-c) (n-c-3)$ and at most
  $(n-c) (n-c-1)$. Given any $(x,y,x',y')\in \Omega^{(4)}$,
  the probability that $g h g^{-1}$ takes $(x,y)$ to $(x',y')$
  equals the probability that $g$ takes $(x,y,x',y')$ to a tuple
  $(r,s,r',s')$ such that $r^h = r'$ and $s^h = s'$, and so lies between
  $(1-\epsilon) (n-c) (n-c-3)/|\Omega^{(4)}|$ and
  $(1+\epsilon) (n-c) (n-c-1)/|\Omega^{(4)}|$.

  Therefore, for any $(x,y)\in \Omega^{(2)}$
  and any subset $S\subset \Omega^{(2)}$, the probability that
  $(x,y)^{g h g^{-1}} \in S$ is at least
  \begin{equation}\label{eq:armador}
    (1-\epsilon) \frac{(n-c) (n-c-3)}{\left|\Omega^{(4)}\right|}
  (|S| - 4 n)\end{equation}
  (Here $|S|- 4 n$ is a lower bound for the number of pairs in $S$
  not containing $x$ or $y$.) 
  We bound (\ref{eq:armador}) from below by
  \[(1-\epsilon) \frac{(n-c) (n-c-3)}{(n-2) (n-3)} 
  \frac{|S|}{n (n-1)} - \frac{4 n (n-c) (n-c-3)}{n (n-1) (n-2) (n-3)}
  \geq \frac{|S|}{\left|\Omega^{(2)}\right|} - \epsilon',\]
  where
  \begin{equation}\label{eq:epsp}\begin{aligned}
  \epsilon' &= \epsilon + 
  \left(1 - \frac{(n-c) (n-c-3)}{(n-2) (n-3)}\right) +
  \frac{4 (n-c) (n-c-3)}{(n-1) (n-2) (n-3)}
  \\
  &\leq \epsilon + 
  \frac{c}{n-3} + \frac{c}{n} + \frac{4 n}{(n-1) (n-2)}  
  \leq \epsilon + \frac{8 + 3 c}{n},\end{aligned}\end{equation}
  by $n\geq 6$. Since we can apply the same bound to the complement
  of $S$, we conclude that the distribution of $(x,y)^{g h g^{-1}}$ is at
  total variation distance at most $\epsilon'$ from the uniform
  distribution. We can apply the same statement to $h^{-1}$ instead of $h$.
  Hence, we can apply
  Lemmas \ref{lem:coeur} and \ref{lem:coeur2} with $g h g^{-1}$ instead
  of $g$, and $\epsilon'$ instead of $\epsilon$, provided that
  their conditions on $\epsilon'$, $\rho$, $m$ and $n$ hold.

  If $n$ is bounded by a constant $C$, say, then we can proceed as follows:
  at every step, we look at an element $L$
  of $Q_j$ of maximal length (size), and let $g$
  be as above, with $\epsilon=1/100$, say. (We could even take $v=\lceil n^5
  \log(n^2/\epsilon)\rceil$, and look only at the Schreier graph
  $\Gamma(G,A;\Omega)$.) Then any given element of $L$ is sent to any
  given element of $\Omega\setminus L$ with positive probability, and so,
  trivially, $L$ becomes larger in $Q_{j+1}= Q_{j} \vee Q_{j}^{g}$ with
  positive probability, i.e., for at least one $g$. We set $g_{j+1}$
  equal to that $g$. After at most $k=C$ steps, we obtain, then, that
  $Q_k$ consists of a single set, equal to $\Omega$, and so we are done.

  We can assume, then, that $n$ is greater than a constant $C$.
  We start by applying parts (\ref{it:turandot1}) and (\ref{it:turandot2}) of Lemma \ref{lem:coeur}
  a bounded number of times so that we get to a state in which
  either the conditions of Lemma \ref{lem:coeur2} are fulfilled
  or $m>\sqrt{n}/50$. (If the conditions are already fulfilled, then, of course, this
  initial stage may be skipped.) The initial value $m_0$ of $m$ will be $2$.
  We let $\rho_0=\rho$.
  We let $\epsilon=\rho/4000$; since we can assume that
  $n\geq 4000 (8+3 c)/\rho$,
  we obtain from (\ref{eq:epsp})
  that $\epsilon' \leq \rho/4000 + (8+3 c)/n
  \leq \rho/2000$. In particular,
  the condition $\epsilon'\leq \rho/100$
  in part (\ref{it:turandot2}) of Lemma \ref{lem:coeur} is satisfied.

  \begin{enumerate}
  \item  We begin by applying part (\ref{it:turandot1}) of Lemma \ref{lem:coeur}
    repeatedly, with $m$ held constant.
  At the $j$th step, Lemma \ref{lem:coeur}
  guarantees us the existence of a $g_j\in A_0^v$ such that,
  for $Q_j = Q_{j-1} \vee Q_{j-1}^{g h g^{-1}}$,
  the proportion $\rho_{j}$ of elements $x\in \Omega$ 
  for which $s_{Q_j}(x)\geq m$ satisfies
  \[\begin{aligned}
  1-\rho_j &\leq (1-\rho_{j-1})^2 + \epsilon'
  \leq 1 - 2 \rho_{j-1} + \frac{\rho_{j-1}}{2000} + \rho_{j-1}^2\\
%  &= 1 - \left(2 - \frac{1}{2} \left(\frac{1}{1000} +\rho_{j-1}\right)\right)
%\rho_{j_1}  + \frac{\rho_{j-1}^2}{2}
&\leq 1 - \frac{3\rho_{j-1}}{2} + \frac{\rho_{j-1}^2}{2} =
(1 -\rho_j) \left(1 - \frac{\rho_j}{2}\right)\end{aligned}\]
  provided that $\rho_{j-1}\leq 999/1000$.
  Thus, letting $j$ be at least about
  \[
  %\frac{\log \frac{1}{1000}}{\log\left(1-\frac{\rho_0}{2}\right)}
  %=
  \frac{\log 1000}{\left|\log\left(1-\frac{\rho_0}{2}\right)\right|}\]
  (which is $O(1/\rho_0)$),
  we obtain a new value $\rho_j$ of $\rho$
  such that $\rho_j\geq 999/1000$,
  say.
\item If $m\geq 1000$, we stop. Otherwise,
we apply part (\ref{it:turandot2}) of Lemma \ref{lem:coeur}.
We are guaranteed the existence of an element $g_{j+1}\in A_0^v$ such that,
for $Q_{j+1} = Q_j \vee Q_j^{g h g^{-1}}$, the proportion of elements $x\in \Omega$ such that
$s_{Q_j}(x)\geq
(1+\rho_j/3) m$ is at least
$\rho_j^2/8$. We choose that $g$, let our new values $m_{j+1}$ and $\rho_{j+1}$
of $m$ and $\rho$ be
\[m_{j+1} = \left\lceil \left(1+\frac{\rho_j}{3}\right) m_j
\right\rceil \geq \left(1 + \frac{333}{1000}\right) m_j
,\;\;\;\;\;  \rho_{j+1} = \frac{\rho_j^2}{8} \geq
\frac{(999/1000)^2}{8} > \frac{1}{9},
\]
and go back to step $1$.
  \end{enumerate}

  It is clear that, after $s = O(1/\rho) + O(\log \max(c,1))$ steps, we obtain a
  partition $Q_s$ such that the proportion of elements
  $x$ of $\Omega$ satisfying $s_{Q_s}(x)\geq 1000 \cdot \max(c,1)$
  is at least $999/1000$.
   
  Now -- and here we are at the heart of the proof of this proposition --
  we go again through an iterative procedure, only we will now
  be alternating a bounded number of applications of
  part (\ref{it:turandot1})
  of
  Lemma \ref{lem:coeur} and an application of Lemma \ref{lem:coeur2}, rather
  than $O(1/\rho)$ applications of part (\ref{it:turandot1}) of Lemma
  \ref{lem:coeur} and
  an application of part (\ref{it:turandot2}) of Lemma \ref{lem:coeur}.
  Throughout the iteration, $\rho$ stays bounded from below by $1/2$.
  We let $\epsilon = 1/n$, so that, by (\ref{eq:epsp}),
  \begin{equation}\label{eq:karton}
    \epsilon' \leq \epsilon + \frac{8 + 3 c}{n} = \frac{9 + 3 c}{n},
  \end{equation}
  and thus $\epsilon'\leq 1/1000$ and
  $\epsilon'\leq 500 \max(c,1)/n \leq \rho m/n$ both hold.
  The conditions of Lemma \ref{lem:coeur2} are thus satisfied for as long
  as $m\leq \sqrt{n}/100$.
    
  The important part this time is that
  Lemma \ref{lem:coeur2} enables
  us to take $m_{j+1} = \lceil m_j^2/2\rceil$, rather than 
  $m_{j+1} = \lceil (1+\rho/3) m_j\rceil$
  as in Lemma \ref{lem:coeur}(\ref{it:turandot2}). Thanks to this fact, after
  only
  $O(\log \log n)$ steps, we obtain a partition $Q_{s'}$ such that the
  proportion of elements $x$ of $\Omega$
  satisfying $s_{Q_{s'}}(x)> \sqrt{n}/100$ is at least $999/1000$.

  We are almost done.
  We apply part (\ref{it:pureplon}) of Lemma \ref{lem:coeur}
  with $m = \lceil \sqrt{n}/100\rceil$, $\rho=999/1000$
  and $\epsilon=1/n$ (and thus $\epsilon'$ as in (\ref{eq:karton}).
  We obtain a partition
  $Q_{s'+1} = Q_{s'} \vee Q_{s'}^g$ containing at least one set of size
  $\geq \min(\rho n/9, (\rho-\epsilon) m^2/2) > ((998/1000)/20000) n > n/20100$.
  Tautologically, the proportion $\rho$ of elements of
  $\Omega$ lying in that set is $>1/20100$.
  
  We then alternate
  a bounded number of applications of part (\ref{it:turandot1}) of Lemma \ref{lem:coeur}
  and an application of part (\ref{it:turandot2}) of the same Lemma,
  iterating a bounded number of times,
  to obtain a partition $Q_{s''}$ such that at least $n/2$ of
  its elements lie in sets of size $>n/2$.

  Finally, we
  apply part (\ref{it:turandot1}) (with $\epsilon=1/n$)
  $O(\log \log n)$ times, and obtain a partition $Q_{s'''}$
  such that the proportion $x$ of elements of $\Omega$ lying in sets
  in the partition of size $>n/2$  is at least $1-2 \epsilon' \geq
  1 - (18+6 c)/n$.
  Since $|\Omega|=n$, there can be at most one set in the partition
  of length $>n/2$;
  that is, at least $n-(18+6 c)$ elements of $\Omega$ lie in one and the same
  set $S$.

  We now proceed as in the case of $n$ bounded, choosing at each step a
  $g\in A_0^l$ that increases the size of the orbit $S$
  by at least $1$. After a bounded number of steps, we obtain that all elements
  of $\Omega$ lie in the same set of the partition, i.e., the final partition
  consists of the single set $\Omega$.
\end{proof}

We come to the main result of this section.
\begin{prop}\label{prop:siniestro}
  Let $\Omega$ be a finite set of size $|\Omega|=n$.
  Let $g_0\in \Sym(\Omega)$ have support of size $\geq \alpha n$,
  $\alpha>0$.
  Let $A\subset \Sym(\Omega)$ with $\langle A\rangle$ $4$-transitive.

  Then there are $\gamma_i\in (A\cup A^{-1} \cup \{e\})^{n^6}$, $1\leq i\leq \ell$, where
  $\ell=O((\log n)/\alpha)$, and
  $g_i\in (A\cup A^{-1}\cup \{e\})^{v}$, $1\leq i\leq k$, $v=O(n^{10})$, $k = O(\log \log n)$, such that, for
  \begin{equation}\label{eq:koppelia}h = \gamma_1 g_0 \gamma_1^{-1} \cdot \gamma_2 g_0 \gamma_2^{-1} \dotsb
  \gamma_\ell g_0 \gamma_\ell^{-1},\end{equation}
  the group
  \begin{equation}\label{eq:kalermo}
    \langle h, g_1 h g_1^{-1}, g_2 h g_2^{-1},\dotsc , g_k h g_k^{-1}\rangle
      \end{equation}
      is transitive.
\end{prop}
\begin{proof}
  Let $\gamma_1,\dotsc,\gamma_\ell$ be as in Lemma \ref{lem:chachava} (applied
  with $g_0$ instead of $g$ and $A\cup A^{-1}\cup \{e\}$ instead of $A$), so that the element $h$ defined in
  (\ref{eq:koppelia}) has support of size $n-c$ with $c=0$ or $c=1$.
  (If $n$ is less than a constant, we do not need apply Lemma
  \ref{lem:chachava}; we can simply let $h = g_0$, as 
  $c = n-|\supp(h)|$ will be bounded.)
  Write $h$ as a product of disjoint cycles, and let
  $P$ be the partition of $\Omega$ given by the cycles.
  
  We can now apply Proposition \ref{prop:verodiro}
  with $\rho = (n-1)/n$ and $c=0$ or $c=1$. It is clear, inductively,
  that, for $0\leq j\leq k$,
  $Q_j$ is finer (not necessarily strictly) than
  the partition of $P$ given by the orbits of
  \[\langle h, g_1 h g_1^{-1}, g_2 h g_2^{-1},\dotsc , g_j h g_j^{-1}\rangle.\]
  Since $Q_k$ is the trivial partition, it follows that the group
  in (\ref{eq:kalermo}) has a single orbit.
\end{proof}

\section{Babai-Seress revisited}\label{sec:babaiseress}
The proof of the main result in \cite{zbMATH00091732}
has what looks like
a bookkeeping mistake, or rather two mistakes, at the very end
(\cite[p. 242]{zbMATH00091732}, ``Proof of Theorem 1.4''): the right
side of the last displayed equation has a factor of $\diam(\Alt(m(G)))$
where it should have a product of squares of several such factors.
We will show how to fix the result and its proof.

(The fact that \cite{zbMATH00091732}
could not be right at this point was first pointed out to the author by L. Pyber; he also shared ideas that can be used in addressing the gap, including one followed and mentioned below.)

The intermediate result \cite[Thm.~2.3]{zbMATH00091732} is actually correct.
We will prove it again here (\S \ref{subs:imptrees}), in part for
the sake of clarity, and in part so as to give an improved version
(Prop.~\ref{prop:obsidian}). We will be following
\cite[\S 3--4]{zbMATH00091732} quite closely.

\subsection{Imprimitivity and structure trees}\label{subs:imptrees}
A {\em tree} is a graph without cycles, with one vertex labeled as
the {\em root vertex}, or {\em root} for short.
The vertices at {\em level} $j$ are those at distance $j$ from the
root. The {\em leaves} of a tree are the vertices at maximal distance from the root.
That maximal distance is called the {\em height} $h$  of the tree.
A {\em child} of a vertex $v$ at {\em level} $j$, $0\leq j<h$, is a vertex at level $j+1$
connected to $v$ by an edge. A {\em descendant} of $v$ is a child of $v$, or a child
of a child, etc.

The following definition has its origin in the study of algorithms on permutation groups,
and in particular \cite{MR973656}. It provides a convenient way to work with
permutation groups that may not be primitive.

\begin{defn}\label{def:whathaf}
  Let $G\leq \Sym(\Omega)$ be a transitive permutation group. A {\em structure tree}
  $T$ for $(G,\Omega)$ is defined as follows. The set of leaves is $\Omega$.
  If $G$ is primitive, then it consists of a root vertex and, for each leaf, an edge
  connecting the root to the leaf. If $G$ is not primitive, we choose
  a maximal block system $B_1,\dotsc,B_k$, define a structure tree for $G$ as a group
  acting transitively on
  $B_1,\dotsc,B_k$, and then draw edges from the vertex corresponding to each $B_i$
  to the elements of $B_i$ (which then become the leaves).
\end{defn}
It is clear that $G$ has a natural action on $T$. Given a vertex $v$ of $T$, we define
the {\em stabilizer}
$G_v \leq  G$ to be the setwise stabilizer of the block corresponding to $v$ (or
the stabilizer of the element of $\Omega$ corresponding to $v$, if $v$ is a leaf).
If $w$ is a descendant of $v$, then $G_w$ is a subgroup of $G_v$.
For $v$ a vertex that is not a leaf, define $K_v$ to be the intersection $\cap_w G_w$,
where $w$ ranges over all children of $v$. It is clear that $K_v$ is a normal subgroup of $G_v$. It is also clear that $G_v/K_v$ acts primitively on the
set of children of $v$, due to the maximality of the block systems used in
Definition \ref{def:whathaf}.

It is easy to see from the definition that $G$ acts transitively on all vertices of $T$
at a given level. It follows that the {\em normal core} 
$N_j = \bigcap_{g\in G} g G_v g^{-1}$ of the stabilizer $G_v$ of a vertex $v$
depends only on the level of $v$. Indeed, it is the intersection $\cap_w G_w$, where
$w$ ranges over all $w$ at the same level as $v$. Moreover,
if the level $j$ of a vertex $v$ is less than the height of the tree
(i.e., $v$ is not a leaf), then the normal core
$\bigcap_{g\in G} g K_v g^{-1}$ of $K_v$ equals $N_{j+1}$.

Part of the reason for working with the groups $G_v$, rather than just with
$N_j$, is that $G_v$ acts transitively on the block
corresponding to $v$, whereas $N_j$ may not.

The following lemmas are quoted as ``folklore'' in \cite[\S 3]{zbMATH00091732}.

\begin{lem}\label{lem:burgunto}
  Let $H$ be a subgroup of a direct product of simple groups
  $M_1\times M_2\times \dotsc \times M_k$ such that the projection $\pi_i:H\to M_i$
  is surjective for every $1\leq i\leq k$.
  Then $H$ is isomorphic to a direct product $\prod_{i\in I} M_i$,
  where $I\subset \{1,2,\dotsc,k\}$.
\end{lem}
\begin{proof}
  If the projection $\phi:H\to M_2\times M_3\times \dotsc \times M_k$ is injective,
  we apply the lemma to $\phi(H)\subset M_2\times \dotsc M_k$ and are done, by
  induction. Suppose $\phi$ is not injective. Let $h_1,h_2\in H$ be distinct elements
  such that $\phi(h_1)=\phi(h_2)$. Then $\phi(h_1^{-1} h_2) = e$, and so
  $h = h_1^{-1} h_2$ lies in $H\cap M_1$. Let $g\in M_1$ be arbitrary. There is an
  element $g'$ of $H$ mapped to $g$ by $\pi_1$; conjugating $h$ by it,
  we obtain $g' h (g')^{-1} = g h g^{-1}$, which must then lie in $H\cap M_1$.
  Since we can do as much for every $g\in M_1$, we conclude that
  $H\cap M_1$ must contain the subgroup of $M_1$ generated by all elements
  of the form $g h g^{-1}$. Since $M_1$ is simple, that subgroup is precisely
  $M_1$, and so $H\cap M_1 = M_1$.

  Let $K = H\cap (\{e\}\times M_2\times \dotsc M_k)$. Since
  $H\cap M_1 = M_1$, we know that $H\sim M_1\times K$.
  For each $2\leq i\leq k$, the image $\pi_i(K)$ is invariant under conjugation
  by $\pi_i(H) = M_i$, and thus must be either $\{e\}$ or $M_i$. We eliminate all
  indices $i$ for which $\pi_i(K) = \{e\}$, and apply the Lemma inductively to
  $K$ as a subgroup of the direct product of the remaining $M_i$.
\end{proof}

\begin{lem}\label{lem:selfcop}
  Let $H_1 \triangleleft H_2 \leq  G$, $H_2/H_1$ simple.
  Let $N_i = \cap_{g\in G} g H_i g^{-1}$. Then $N_2/N_1$ is isomorphic to a direct
  product of copies of $H_2/H_1$.
\end{lem}
\begin{proof}
  We may assume that $N_2\ne N_1$.
  By the second isomorphism theorem, $H_1\cap N_2$ is a normal
  subgroup of $H_2\cap N_2 = N_2$, and $N_2/(H_1\cap N_2)$ is isomorphic to
  $N_2 H_1/ H_1$, which is a normal subgroup of $H_2/H_1$. Since $H_2/H_1$ is
  simple, that subgroup is either trivial or all of $H_2/H_1$. If it were trivial, then
  $N_2 \leq  H_1$, and so $N_2 \leq  g H_1 g^{-1}$ for every
  $g\in G$; it would follow immediately that $N_2 = N_1\cap N_2 = N_1$. Thus,
  we may assume that $N_2/(H_1\cap N_2)$ is isomorphic to $H_2/H_1$.

  The same argument applied to any conjugate $g H_1 g^{-1}$, $g\in G$,
  instead of $H$ shows that
  $N_2/(g H_1 g^{-1} \cap N_2)$ is isomorphic to $H_2/H_1$.
  Now, $N_2/N_1$ is isomorphic to its image under the natural map
  $N_2/N_1 \to (N_2/(g H_1 g^{-1}\cap N_2))_{g\in G}$, since $N_1 = \bigcap_{g\in G}
  (g H_1 g^{-1}\cap N_2)$. We apply Lemma \ref{lem:burgunto}, and obtain that
  $N_2/N_1$ is isomorphic to a direct product of groups of the form
  $N_2/(g H_1 g^{-1}\cap N_2) \sim H_2/H_1$.
\end{proof}

The following is an extremely useful (and by now standard) consequence of the Classification
Theorem and the O'Nan-Scott theorem.
This is the one way in which the Classification Theorem is needed for
the proof of our results.
\begin{prop}[\cite{MR599634}, \cite{MR758332}]\label{prop:camlie}
  %\footnote{For more details on the action
  %  of $G$ on $\Omega$, see \cite{Lie}.
  %  There is also a statement
  %  with precise constants due to \cite{Ma}, but we will not need it.}
  Let $G\leq \Sym(\Omega)$ be a primitive group, where $|\Omega|=n$. Then either
  \begin{enumerate}
  \item\label{it:uno} $|G|\leq n^{O(\log n)}$, or
  \item\label{it:duo}
    there is a subgroup $N\triangleleft G$, $\lbrack G:N\rbrack\leq n$,
    isomorphic to a direct product $\Alt(m)^r = \Alt(m)\times \dotsc \times \Alt(m)$, where $r\geq 1$, $m\geq 5$ and
    $n = \binom{m}{k}^r$
    for some $1\leq k\leq m-1$.
  \end{enumerate}
\end{prop}
\begin{proof}
  Just a few words on how the statement follows from
  \cite[Main Thm.]{MR758332}.
  Case (ii) there asserts that there
  is a set $\Delta\subset \Omega$ with $|\Delta|<9 \log_2 n$ such that
  $G_{(\Delta)}=\{e\}$; since $\lbrack G: G_{(\Delta)}\rbrack\leq n^{|\Delta|}$,
  it follows immediately that conclusion
  (\ref{it:uno}) holds.

  Assume, then, that we are in case (i) in \cite[Main Thm.]{MR758332};
  that case gives us
  a subgroup $N$ as in (\ref{it:duo}) here. It also gives us that
  $m\geq 2$, $\lbrack G:N\rbrack \leq 2^r r!$
  and $n = \left(\binom{m}{k}\right)^r$ for some $1\leq k\leq m-1$.
  Clearly, $n\geq m^r$.
  
  If $m\geq \max(2 r,5)$, then $\lbrack G: N\rbrack \leq (2 r)^r \leq m^r\leq n$,
  and so we obtain conclusion (\ref{it:duo}). If $m<2 r$, then,
  since $r\leq \log_m n\leq \log_2 n$, we see that
  $|N| = (m!/2)^r\leq m^{m r} < n^m < n^{\max(2 r,5)} \leq n^{O(\log n)}$, whereas
  $\lbrack G:N\rbrack \leq 2^r r!\leq (2 r)^r \leq n^{O(\log n)}$.
  Hence $|G| = n^{O(\log n)}$, that is, conclusion (\ref{it:uno}) holds.
\end{proof}
  The motivation for wanting subgroups $\Alt(m)$ with $m\geq 5$ is of course
  that $\Alt(m)$ is then simple.
  
We could do without the following bound\footnote{Thanks are due to D. Holt
  for the reference.}, in that using a trivial bound in case (\ref{it:duo})
of Prop.~\ref{prop:camlie} would be enough for our purposes; our
intermediate results would become somewhat weaker, but our final results 
(Thms.~\ref{thm:jukuju} and \ref{thm:molop}) would not be affected.
At the same time, there is no
reason to avoid the lemma we are about to state. It does use the Classification
Theorem, but only in the sense that it uses \cite[Main Thm.]{MR758332}
(or rather the version with sharp constants in \cite{MR1943938}).
\begin{lem}\label{lem:gprv}(\cite[Thm.~1.3]{gprv})
  Let $G\leq \Sym(\Omega)$ be a primitive group, where $|\Omega|=n$.
  Let $\{e\} = H_0 \triangleleft H_1 \triangleleft\dotsc \triangleleft H_\ell=G$. Then
  \[\ell \leq \frac{8}{3} \frac{\log n}{\log 2} - \frac{4}{3}.\] 
\end{lem}
The trivial bound assuming Prop.~\ref{prop:camlie}, 
would be
$\ell \ll (\log n)^2$.

\begin{lem}[\cite{MR860123}]\label{lem:babsub}
  The length $\ell$ of any subgroup chain $\{e\} = H_0 \lneq H_1 \lneq \dotsc \lneq H_\ell
  = \Sym(n)$ of $\Sym(n)$ is at most $2 n-3$.
\end{lem}
%The bound given by \cite{MR860123} is in fact $2 n - 3$.
The trivial bound would be $\ell \leq (\log |\Alt(n)|)/\log 2 = O( n \log n)$.

We will now prove a version of \cite[Thm.~2.3]{zbMATH00091732}.
\begin{prop}\label{prop:obsidian}
  Let $G\leq \Sym(\Omega)$ be transitive, $|\Omega|=n$. Then $G$ has a series of normal subgroups
  $\{e\} = H_0 \triangleleft H_1 \triangleleft \dotsc \triangleleft H_\ell = G$,
  $H_i\triangleleft G$, with $\{1,\dotsc,\ell\}$ being partitioned into two sets,
  $A$, $B$, such that these properties hold:
  \begin{enumerate}
  \item\label{it:obsprod} each quotient $H_{i+1}/H_i$ is a direct product of at most $2 n$ copies of
    a simple group $M_i$,
  \item\label{it:obsii} for each $i\in A$, the group $M_i$ is an alternating group $\Alt(m_i)$ with $m_i\geq 5$,
  \item\label{it:obsiii} $m = \prod_{i\in A} m_i$ satisfies $m\leq n$,
  \item\label{it:obs3.5} if $G\ne \Alt(\Omega)$ and $G\ne \Sym(\Omega)$, then
    $m_i\leq n/2$ for every $i\in A$,
  \item\label{it:obsiv} $\prod_{i\in B} |M_i| = (n/m)^{O(\log(n/m))} \cdot m = n^{O(\log n)}$,
  \item\label{it:obsv} $\ell = O\left(\log n\right)$.
  \end{enumerate}
  All implied constants are absolute.
\end{prop}
The series $\{e\} = H_0 \triangleleft H_1 \triangleleft \dotsc \triangleleft H_\ell = G$ 
is a refinement of the series $\{e\} = N_1 \triangleleft \dotsc \triangleleft N_h = G$
defined by the structure tree as above.
\begin{proof}
  We construct a structure tree as in Def.~\ref{def:whathaf}.
  We choose a leaf $v$ and denote by $v_0, v_1,\dotsc, v_{h-1}, v_h = v$ all
  vertices on the path from the root $v_0$ to $v$. For $0\leq i\leq h-1$,
  let $G_i = G_{v_i}/K_{v_i}$.

  Apply Prop.~\ref{prop:camlie} with $G_i$ instead of $G$, and with the set
  of children of $v_i$ instead of $\Omega$. Write $n_i$ for the number of
  children of $v_i$. Clearly, $\prod_{0\leq i\leq h-1} n_i = n$, and so
  $h \leq (\log n)/\log 2 = O(\log n)$.
  
  If conclusion (\ref{it:uno})
  holds, we simply take
  a composition series
\[\{e\} = S_{i,0} \triangleleft S_{i,1} \triangleleft \dotsc \triangleleft S_{i,\ell_i} = G_i\]
  of $G_i$. By Lemma \ref{lem:gprv},
  $\ell_i \ll \log n_i$. Write $\tilde{S}_{i,j}$ for the preimage of
  $S_{i,j}$ under the map $G_{v_i} \to G_i = G_{v_i}/K_{v_i}$. Let
  \[H_{i,j} = \bigcap_{g\in G} g \tilde{S}_{i,j} g^{-1}.\]
  By Lemma \ref{lem:selfcop}, for
  $1\leq j\leq \ell_i$, $H_{i,j}/H_{i,j-1}$ is a direct product of copies
  of the simple group $M_{i,j} = \tilde{S}_{i,j}/\tilde{S}_{i,j-1} = S_{i,j}/S_{i,j-1}$.
  By Lemma \ref{lem:babsub}, there are at most $2 n$ such copies. Evidently,
  \[\prod_{1\leq j\leq \ell_i} \left|M_{i,j}\right| = |G_i| = n_i^{O(\log n_i)}.\]
  We include every one of the groups $M_{i,j}$ in the set $B$ (to be implicitly
  redefined as a set of indices at the end of the proof).

  If conclusion (\ref{it:duo}) of Prop.~\ref{prop:camlie} holds, we
  write $r_i$ for $r$, $m_i$ for $m$ and $k_i$ for $k$, and let
  \[\{e\} = S_{i,r_i} \triangleleft S_{i,r_i+1} \triangleleft \dotsc \triangleleft S_{i,\ell_i} = G_i/N\]
be a composition series of $G_i/N$. Since $|G_i/N|\leq n_i$, we see that
$\ell_i = r_i + O(\log n_i) = O(\log n_i)$.
Given that $N\sim \Alt(m_i)^{r_i}$, we can write
\[\{e\} = A_{i,0} \triangleleft A_{i,1}\triangleleft \dotsc \triangleleft
A_{i,r_i} = N,\]
where $A_{i,j}/A_{i,j-1} \sim \Alt(m_i)$ for $1\leq j\leq r_i$.
This time, we define
$\tilde{A}_{i,j}$ to be the preimage of
$A_{i,j}$ under the map $G_{v_i} \to G_i = G_{v_i}/K_{v_i}$, and
$\tilde{S}_{i,j}$ to be the image of 
  $S_{i,j}$ under the composition $G_{v_i} \to G_i \to G_i/N$. 

Let
\[H_{i,j} = \begin{cases} \bigcap_{g\in G} g \tilde{A}_{i,j} g^{-1} &\text{if $0\leq j\leq r_i$,}\\
  \bigcap_{g\in G} g \tilde{S}_{i,j}  g^{-1}
  &\text{if $r_i < j\leq \ell_i$.}\end{cases}\]
We let $M_{i,j} = S_{i,j}/S_{i,j-1}$ for $r_i < j\leq \ell_i$, and
$M_{j,1} = A_{i,j}/A_{i,j-1} \sim \Alt(m_i)$ for $1\leq j \leq r_i$.

We know from Prop.~\ref{prop:camlie} that $\binom{m_i}{k_i}^{r_i}\leq n_i$,
where $1\leq k_i\leq m_i-1$.
If $r_i\geq 2$, then $m_i\leq \lfloor \sqrt{n_i}\rfloor \leq n_i/2\leq n/2$;
if $r_i=1$ and $k_i\geq 2$, then $m_i (m_i-1) \leq 2 n_i$ and so, since $m_i\geq 5$,
$m_i\leq n_i/2\leq n/2$. If $r_i=1$ and $k_i=1$, then
$m_i = n_i$.
In that case, if $G$ is not primitive, then $m_i=n_i\leq n/2$, whereas,
if $G$ is primitive, $m_i=n_i=n$ and so $\Alt(n)\leq G$.

For all $1\leq j\leq \ell_i$, $M_{i,j}$ is simple.
  By Lemma \ref{lem:selfcop}, for $1\leq j\leq \ell_i$,
  $H_{i,j}/H_{i,j-1}$ is a direct product of copies
  of the simple group $M_{i,j}$;
  by Lemma \ref{lem:babsub},
  there are at most $2 n$ such copies.
  We include $M_{i,j} \sim \Alt(m_i)$ in $A$ for
  $1\leq j\leq r_i$ 
  and $M_{i,j}$ in $B$ for $r_i<j\leq \ell_i$.

  It is clear -- whether conclusion (\ref{it:uno}) or (\ref{it:duo}) holds --
  that, for any $i$ less than the height $h$ of our tree,
  $H_{i,0} = \cap_{g\in G} g K_{v_i} g^{-1} = N_{i+1}$, whereas
  $H_{i,\ell_i} = \cap_{g\in G} g G_{v_i} g^{-1} = N_i$. Hence we can define
  our subgroups $H_0,H_1,\dotsc,H_\ell$
  to be
  \[H_{h-1,0},H_{h-1,1},\dotsc ,
  H_{h-1,\ell_{h-1}} = H_{h-2,0},\dotsc,
  H_{h-2,\ell_{h-2}}, \dotsc H_{1,\ell_1} =
  H_{0,0}, \dotsc , H_{0,\ell_0}.\]
  The (trivial) accounting is left to the reader.

\end{proof}

\subsection{Reduction of the diameter problem to the case of alternating groups}

First, a lemma essentially due to Schreier. The statement is as in
\cite[Lemma 5.1]{zbMATH00091732}.
\begin{lem}[Schreier]\label{lem:schreier}
  Let $G$ be a finite group. Let $N\triangleleft G$. Then
  \[\begin{aligned}
  \diam G &\leq (2 \diam(G/N) + 1) \diam(N) + \diam(G/N)\\
  &\leq 4 \diam(G/N) \diam(N).\end{aligned}\]
\end{lem}
\begin{proof}
  Let us be given a set $A = \{g_1,\dotsc,g_r\}$ of generators of $G$. Write
  $d_1$ for $\diam(G/N)$, $d_2$ for $\diam(N)$ and $m$ for $|G/N|$.
  Then, by the definition of diameter, there are
  $\sigma_1,\dotsc,\sigma_m \in (A\cup A^{-1}\cup \{e\})^{d_1}$ giving us
  a full set of representatives of $G/N$. As is well-known,
  \[S = \{\sigma_i g_j \sigma_k^{-1} : 1\leq i,k\leq m, 1\leq j\leq r\}\cap N\]
  is a set of generators of $N$ ({\em Schreier generators}).
  Hence, $N = (S \cup S^{-1} \cup \{e\})^{d_2}$, and so
  \[\begin{aligned}
  G = \{\sigma_1,\dotsc,\sigma_m\}\cdot N &\subset
  (A\cup A^{-1}\cup \{e\})^{d_1} \cdot (S\cup S^{-1} \cup \{e\})^{d_2}\\
  &\subset (A \cup A^{-1}\cup \{e\})^{d_1 + (2 d_1 + 1) d_2}.\end{aligned}\]
\end{proof}
\begin{cor}\label{cor:moreschreier}
  Let $G$ be a finite group. Let $\{e\} \triangleleft H_1\triangleleft H_2
  \triangleleft \dotsc \triangleleft H_\ell = G$. Then
  \[\diam(G)\leq 4^{\ell-2} \prod_{i=0}^{\ell-1} \diam\left(H_{i+1}/H_i\right)
.\]
\end{cor}
\begin{proof}
  Immediate from Lemma \ref{lem:schreier}.
\end{proof}

\begin{lem}\label{lem:bspyb0}(\cite[Lemma 5.4]{zbMATH00091732})
  Let $G = T_1\times T_2\times \dotsb \times T_n$, where the $T_i$ are non-abelian
  simple groups. Let $\diam(T_i) = d_i$, $d =\max_i d_i$. Then
  $\diam(G)\ll n^3 d^2$.
\end{lem}
We will go over the ideas of the proof of the Lemma in a moment,
when we improve it in
the special case of the alternating group (Lemma \ref{lem:bspyb}).
L. Pyber pointed out to the
author that the dependence on $d$ could and should be improved so as to be
linear; the quadratic dependence of Lemma \ref{lem:bspyb0} on $d$ is
one of the gaps in the proof of the main
result in \cite{zbMATH00091732}.
 It is actually enough to improve the dependence on $d$ in the alternating case.
The tool we will use is a simple lemma, similar
to \cite[Prop.~5.8]{zbMATH00091732}.

\begin{lem}\label{lem:supconj}
  Let $g\in \Alt(\Omega)$, $g\ne e$, $|\Omega|\geq 4$. Then there is an
  $h\in \Alt(\Omega)$ such that $\lbrack g,h \rbrack$
  is either a $3$-cycle or a product of two disjoint $2$-cycles.
\end{lem}
Here $\lbrack g,h\rbrack$ denotes the commutator $g^{-1} h^{-1} g h$.
\begin{proof}
  Write $g$ as a product of disjoint cycles. If
  $g$ contains two disjoint $3$-cycles $(a b c)$, $(d e f)$,
  let $h$ equal $(a d c) (b e f)$. Then
  $\lbrack g,h\rbrack = (a f) (b d)$.
  If $g$ contains two disjoint $2$-cycles $(a b) (c d)$, let $h$
  be the $3$-cycle $(a b c)$; their commutator
  $\lbrack g,h\rbrack$ 
  will be $(a c) (b d)$.
  If $g$ contains a $k$-cycle $(a b c d\dotsc)$, $k\geq 4$, let
  $h = (a b c)$. Then $\lbrack g,h\rbrack = (a d c)$.
  Finally, if $g$ consists of a single $3$-cycle $(a b c)$, let
  $d$ be an element of $\Omega$ different from $a$, $b$ and $c$, and
  define $h$ to be $(b c d)$. Then $\lbrack g,h\rbrack = (a d) (b c)$.
\end{proof}

The following lemma is of course extremely familiar.
\begin{lem}\label{lem:easypeasy}
  Let $n\geq 5$. Then
  every element of $\Alt(\Omega)$, $|\Omega|=n$, can be written as
  (a) the product of at most $n-1$ $3$-cycles, (b) the product of at most
  $(n+1)/2$ elements of the form $(a b) (c d)$.
\end{lem}
\begin{proof}
  We prove part (a) by induction. If $g\in \Alt(\Omega)$ is not the identity,
  then there is an $a\in \Omega$ such that $b = a^g$ is distinct from $a$,
  and another $c\ne a, b$ that is also in the support of $g$. Then, for
  $g' = g \cdot (b a c)$, we see that $a^{g'} = a$ and $\supp(g') \subset
  \supp(g)$, and so $|\supp(g')|\leq |\supp(g)|-1$.

  We prove part (b) in the same way: if $|\supp(g)|\geq 4$, there
  are distinct $a,b,c,d$ such that $a^g = b$ and $c^g = d$; then, for
  $g' = g \cdot (b a) (d c)$, $|\supp(g')|\leq |\supp(g)|-2$.
  If $|\supp(g)|=3$, then $g$ is a $3$-cycle $(a b c)$, and so, for
  $b', c'$ not in the support of $g$, $g$ equals the product of
  $(a c) (b' c')$ and $(b c) (b' c')$.
\end{proof}

The following result is classical, easy and very well-known. According
to \cite{MR2331612}, it was first proved in \cite{miller1899commutators}.
\begin{lem}\label{lem:oldmiller}
  Let $m\geq 5$. Then every element of $\Alt(m)$ is a commutator, i.e.,
  expressible in
  the form $\lbrack x,y\rbrack$, $x,y\in \Alt(m)$.
\end{lem}

Now we come to the proof of an improved version of Lemma \ref{lem:bspyb0}
in the case of the alternating group.
\begin{lem}\label{lem:bspyb}
  Let $G = T_1\times T_2\times \dotsb \times T_n$, where the $T_i$ are
  alternating groups $\Alt(m_i)$, $m_i\geq 5$. Let $\diam(T_i) = d_i$, $d =\max_i d_i$,
  $m = \max_i m_i$. Then
  $\diam(G)\ll n^3 m d$.
\end{lem}
L. Pyber suggests using \cite{MR1865975} to prove an analogous improvement
on Lemma \ref{lem:bspyb0} for
arbitrary finite simple groups $T_i$.
\begin{proof}
  Let $S$ be a set of generators of $G$, and let $A = S\cup S^{-1} \cup \{e\}$.
  Write $\pi_i:G \to T_i$ for the projection of $G$ to $T_i$.
  
  By the
  definition of $d$, $\pi_i\left(A^d\right) = T_i$ for every $1\leq i\leq m$.
  The set $A^{2 d +1}\cap \ker(\pi_i)$ must then contain a set of generators
  of $\ker(\pi_i)$ (namely, Schreier generators). In particular, for any
  $j\ne i$, $A^{2 d +1}\cap \ker(\pi_i)$ contains at least one element
  $g_{i,j}$ such that $\pi_j(g_{i,j})\ne e$. By Lemma \ref{lem:supconj}
  and $\pi_j\left(A^d\right) = T_j$, there is an $h\in A^d$ such that
  $\pi_j(\lbrack g, h\rbrack)\in \pi_j(A^{6 d + 2})$
  is either a $3$-cycle or the product of two
  disjoint $2$-cycles. Hence, conjugating
  $\lbrack g,h\rbrack$ by all elements of $A^d$,
  we obtain either all
  $3$-cycles in $T_j$ or all products of disjoint $2$-cycles in $T_j$.
  By Lemma~\ref{lem:easypeasy}, we can express every element of $T_j$
  as a product of (a) at most $m_j-1$ $3$-cycles, (b) at most
  $(m_j+1)/2$ $3$-cycles.
  At the same time, $\lbrack g, h\rbrack$ is in $\ker(\pi_i)$, and so,
  obviously,
  are its conjugates.
  Hence, for $B_i = A^{(8 d + 2) m_j}\cap \ker(\pi_i)$, we see that
  $\pi_j(B_i) = T_j$.

%  (When applying Lemma \ref{lem:easypeasy}, we tacitly assumed that
%  $m_j\geq 5$. If $m_j<5$, then $\diam(T_j) < 2 m_j$ (use Lemma
%  \ref{lem:schreier} if $m_j=4$), and so $A^{(4 d + 2) m_j}\cap \ker(\pi_i)$
%  projects on $T_j$; hence, $\pi_j(B_i) = T_j$, just as in the general case.)
  
  Now, say that, for $S,S'\subset \{1,\dotsc,n\}\setminus \{j\}$,
  there are sets $B_{S}, B_{S'}\subset A^{k}$ satisfying
  $\pi_j(B_S) = \pi_j(B_{S'}) = T_j$ as well as $B_S\subset \ker(\pi_i)$ for every
  $i\in S$ and $B_{S'} \subset \ker(\pi_i)$ for every $i\in S'$.
  Then $B_{S\cup S'} =\{\lbrack x,y\rbrack: x\in B_S, y\in B_{S'}\}$ is
  a subset of $A^{4 k}$ contained in $\ker(\pi_i)$ for every $i\in S\cup S'$.
  Moreover, by Lemma \ref{lem:oldmiller}, $\pi_j(B_{S\cup S'}) = T_j$.

  We apply this procedure repeatedly, first expressing
  $Z_j = \{1,\dotsc,n\}\setminus \{j\}$ as the union of two disjoint sets
  $S$, $S'$ of size $\lfloor (n-1)/2\rfloor$ and $\lceil (n-1)/2 \rceil$,
  respectively,
  and then doing a recursion, expressing at each point the set we are given
  as the union of two disjoint sets of sizes differing by at most $1$,
  until we reach the single-element sets $S = \{i\}$, $i\ne j$.
  We obtain a subset $B_{Z_j}$ of $A^{4^{\lceil \log_2 n\rceil} k} \subset A^{4 n^2 k}$,
  where $k = (8 d + 2) m_j$, such
  that $B_{Z_j} \subset \ker(\pi_i)$ for every $i\ne j$, and
  $\pi_j(B_{Z_j}) = T_j$. (Here we note that $4^{\lceil \log_2 n\rceil} \leq
  4^{\log_2 n + 1} \leq 4 \cdot 4^{\log_2 n} = 4 n^2$.)

  Multiplying the sets $B_{Z_j}$, we obtain that
  $A^{4 n^3 (8 d + 2) m}$ contains all of $G$.
\end{proof}

We will also need a very easy analogue for {\em abelian} simple groups.
\begin{lem}\label{lem:bspybtr}
  Let $G = \mathbb{Z}/p\mathbb{Z} \times \mathbb{Z}/p\mathbb{Z} \times \dotsb \times \mathbb{Z}/p\mathbb{Z}$ ($n$ times). Then
  $\diam(G)\leq n \lfloor p/2\rfloor$.  
\end{lem}
\begin{proof}
  Let $S$ be a set of generators of $G$, and let $A = S\cup S^{-1} \cup \{e\}$.
   We can see $G$ both as a group
   and as a vector space (which we may call $V$)
   over $\mathbb{Z}/p\mathbb{Z}$.
 We choose a non-identity element $v$ of $A$.

  Trivially, every element of the linear span
  $\langle v\rangle$ of $v$ can be written in the group
  as $v^j$ for some $j\in \mathbb{Z}$ with
  $|j|\leq (p-1)/2$. We project the elements of $A$ to $A \mod \langle v\rangle$, and thus reduce the problem to that for the space $V \mo \langle v\rangle$ instead of
  $V=G$.
\end{proof}
We finally come to the fixed (and improved) version
of \cite[Thm.~1.4]{zbMATH00091732}.
%The original version has what looks like a bookkeeping mistake in the last steps.
\begin{prop}\label{prop:finbo}
  Let $G\leq \Sym(\Omega)$, $|\Omega|=n$, be transitive. 
  Then there are
  $m_1,\dotsc,m_k\geq 5$ with $\prod_{i=1}^k m_i\leq n$ such that
  \[\diam(G) \leq n^{O(\log n)} \prod_{i=1}^k \diam(\Alt(m_i)).\]
  Moreover,
  \begin{enumerate}
   \item for every $1\leq i\leq k$, $G$ has a composition
     factor isomorphic to $\Alt(m_i)$,
   \item if $G\ne \Alt(\Omega), \Sym(\Omega)$, then
     $m_i\leq n/2$ for every $1\leq i\leq k$.
  \end{enumerate}
\end{prop}
\begin{proof}
  Apply Proposition \ref{prop:obsidian}. By Cor.~\ref{cor:moreschreier},
  \[\diam(G) \leq 4^{O(\log n)} \prod_{i=0}^{\ell-1} \diam(H_{i+1}/H_i).\]
  By Lemma \ref{lem:bspyb0}, Lemmas \ref{lem:bspyb}--\ref{lem:bspybtr}
  and Prop.~\ref{prop:obsidian}(\ref{it:obsprod}),  
  \[\diam(H_{i+1}/H_i) \ll \begin{cases}
    n^3 m_i \diam(M_i) &\text{if $i\in A$,}\\
n^3 \diam(M_i)^2 &\text{if $i\in B$.}\end{cases}\]
  Hence
  \[\diam(G)\leq 4^{O(\log n)} (O(n^3))^\ell
  \prod_{i\in A} \left(m_i \diam(M_i)\right) \prod_{i\in B} \diam(M_i)^2.\]
  Trivially, $\diam(M_i)\leq |M_i|$, and so,
  by  Prop.~\ref{prop:obsidian}(\ref{it:obsiv})
  \[\prod_{i\in B} \diam(M_i) \leq
  \prod_{i\in B} \left|M_i\right| \leq n^{O(\log n)},\]
  whereas, by Prop.~\ref{prop:obsidian}(\ref{it:obsii}),
  \[\prod_{i\in A} \diam(M_i) = \prod_{i\in A} \diam(\Alt(m_i)).\]
  Finally, by
  Prop.~\ref{prop:obsidian}(\ref{it:obsiii}) and (\ref{it:obsv}),
  $\prod_{i\in A} m_i\leq n$ and $\ell \ll \log n$.
\end{proof}

\section{Main argument}
Let us set out to prove our main result (Theorem \ref{thm:jukuju}).
Part of the general strategy will be as in \cite{MR3152942},
but much simplified.

Throughout,
$A\subset \Sym(\Omega)$, $|\Omega|=n$, with $A=A^{-1}$, $e\in A$, and
$\langle A\rangle$ is $3$-transitive. We assume that
$\log |A| \geq C (\log n)^3$, where $C>0$ is a constant large enough for our later uses.

\subsection{Existence of a large prefix}

For $j=1,2,\dotsc$, we choose distinct elements
$\alpha_1, \alpha_2,\dotsc \in \Omega$ such that, for
$j=1,2,\dotsc$,
\begin{equation}\label{eq:likemula}
  \left|\alpha_j^{(A^4)_{(\alpha_1,\alpha_2,\dotsc,\alpha_{j-1})}}\right|\geq \rho n,\end{equation}
where we set $\rho = e^{-1/5} = 0.818\dotsc$, say. 
We stop when
$(A^4)_{(\alpha_1,\alpha_2,\dotsc,\alpha_k)}$ has no orbits of size $\geq \rho n$.
Inequality (\ref{eq:likemula}) holds for $1\leq j\leq k$.

Let $\Sigma = \{\alpha_1,\dotsc,\alpha_{k-1}\}$. By Corollary \ref{cor:ratherbab} (applied with $\Sigma \cup \alpha_k$ instead of $\Sigma$),
either (\ref{eq:secondopt}) holds, and we are done, or
\begin{equation}\label{eq:amianto}
  k \geq \frac{\log |A|}{15 (\log n)^2}.
\end{equation}
We can assume henceforth that (\ref{eq:amianto}) holds.

By Lemma \ref{lem:basilic}, the restriction
$\left(A^{8 (k-1)}\right)_{\Sigma}|_\Sigma$ is a subset of $\Sym(\Sigma)$
with at least $\rho^{k-1} (k-1)!$ elements.
Let $A' = \left(A^{8 (k-1)}\right)_{\Sigma}$, $H = \left\langle A'\right\rangle$.
If, as we may assume, $k$ is larger than an absolute constant, then,
by Lemma \ref{lem:amusi}, there exists an orbit $\Delta\subset \Omega$ of
$H|_\Sigma$, 
such that $|\Delta|\geq \rho\cdot (k-1)$ and $H|_\Delta$
contains $\Alt(\Delta)$. Thus, in particular,
$H$ has a section
isomorphic to $\Alt(k-1)$, namely, the quotient defined by restricting
either $H$ or a subgroup of $H$ of index $2$ to $\Delta$.

\subsection{The case of descent}

Applying Lemma \ref{lem:wielandt} with $\epsilon=1/8$, we see that
$H$ contains an element $g_0\ne e$ such that $|\supp(g_0)|<n/8$,
assuming, as we may, that $n$ is greater than an absolute constant and that
$k-1\geq C_2 \log n$, where $C_2$ is an absolute
constant.

Let $O = \alpha_k^{(A^4)_{(\Sigma)}}$. We know that
(\ref{eq:likemula}) holds for $j=k$, i.e., $|O|\geq \rho n$.
Denote by $O'$ the orbit of $H$ containing $O$.
%Let us say a group is ``colossal'' if it is isomorphic to
%$\Alt(m)$ or $\Sym(m)$ for some $m\geq e^{-1/10} \cdot n$.
%(The nomenclature is motivated by that of Babai, who calls $\Alt(m)$
%and $\Sym(m)$ ``giants''.)
%The question now is whether $H$ acts like a colossal group on $O'$.

Suppose first that either $|O'|\leq e^{-1/10} n$ or 
$H|_{O'}\ne \Alt(O'), \Sym(O')$.
Define $D$ to be the diameter of
$H|_{O'}$. If the diameter of $\Gamma(H,A')$ is no larger than $D$, then
$g_0\in (A')^D \subset A^{8 k D}$. 

If, on the other hand,
$\diam \Gamma(H,A')> D$, then there is an element $g$ of $H$ that is in
$(A')^{D+1}$ but not in $(A')^D$. At the same time, since $D$ is the diameter
of $H|_{O'}$, there is an element $h$ of $(A')^D$ whose restriction
$h|_{O'}$ equals $g|_{O'}$. Clearly, $g^{-1} h$ is non-trivial, lies in
$(A')^{2 D + 1}$ and has trivial restriction to $O'$. By this last fact,
the support of $g^{-1} h$ is of size $\leq (1-\rho) n < n/5$. 

Therefore, in either case, there exists an element $g'$ of $(A')^{2 D+1}\subset
A^{8 k (2 D +1)}$ with support $< n/5$. (Many thanks are due to Henry Bradford for spotting a gap at this point in a previous version of this paper, and for the alternative argument we have just given.)
By Lemma \ref{lem:bbs},
\[\diam(\Gamma(\langle A\rangle,A\cup g')) \ll n^8 (\log n)^{O(1)},\]
and so
\[\diam(\Gamma(\langle A\rangle,A))\ll 8 k (2 D + 1) n^8 (\log n)^{O(1)} \ll n^{10} D.\]
Thus we attain conclusion (\ref{eq:rororo}) in Theorem \ref{thm:jukuju}.
We call this case the case of {\em descent}.

\subsection{The case of growth}

Assume henceforth that $H|_{O'}$ is either $\Alt(O')$ or $\Sym(O')$ and
that $|O'|\geq e^{-1/10} n$.
We can then of course assume that
$|O'|\geq 6$, and so the action of $H$ on $O'$ is $4$-transitive. 
%Then the action of $H_{\alpha_k}$ on $O'' = O'\setminus (O'\cup \{\alpha_k\})$
%is $2$-transitive.
%By Cor.~\ref{cor:schreicor}, $H_{\alpha_k} = \langle A''\rangle$
%for
%\begin{equation}\label{eq:almagest}
%  A'' =
%  \left(\left(\left(A^{8 k}\right)_{\Sigma}\right)^{3 |\Omega'|-3}\right)_{\alpha_k} .
%  \end{equation}
%Clearly, $A''\subset \left(A^{24 k n}\right)_{\{\alpha_1,\dotsc,\alpha_k\}}$.

Let $B = (A^2)_{(\alpha_1,\alpha_2,\dotsc,\alpha_k)}$. Then, by (\ref{eq:likemula}),
every orbit of $B B^{-1}$
is of length  $< \rho n = e^{-1/5} n \leq e^{-1/10} |O'|$.
We apply  Cor.~\ref{cor:lalmo}
with $A'|_{O'}$ instead of $A$, $B|_{O'}$ instead of $B$,
$O'$ instead of $\Omega$,
$|O'|$ instead of $n$ and $e^{-1/10}$ instead of $\rho$, and obtain
that there is a $g\in (A')^m \subset A^{8 k m}$, $m\ll n^6 \log n$, such that
\[\left|B^2 g B^2 g^{-1}\right|\geq |B|^{1 + \frac{1/10}{\log n}}.\]
Since $g$ is in the setwise stabilizer of $\Sigma$,
we know that
$B^2 g B^2 g^{-1}$ is a subset of $\left(A^{16 k m + 8}\right)_{(\Sigma)}$.
Therefore, by Lemma \ref{lem:duffy},
\[\left|\left(A^{32 k m + 16}\right)_{(\alpha_1,\dotsc,\alpha_k)}\right|
\geq \frac{\left|\left(A^{16 k m + 8}\right)_{(\Sigma)}
\right|}{n} \geq \frac{1}{n} |B|^{1 + \frac{1/10}{\log n}}.\]

We have obtained what we wanted: growth in a subgroup -- namely, the
subgroup $\Sym(\Omega)_{(\alpha_1,\dotsc,\alpha_k)}$ of $\Sym(\Omega)$. We apply Lemma
\ref{lem:durdo} with $\Sym(\Omega)_{(\alpha_1,\dotsc,\alpha_k)}$ instead of $H$
and $32 k m + 16$ instead of $k$, and obtain that
\begin{equation}\label{eq:serdukt}
  \left|A^{32 k m + 17}\right|\geq \frac{|B|^{\frac{1/10}{\log n}}}{n} |A|.
  \end{equation}

Only one thing remains: to ensure that $|B|$ is not negligible
compared to $|A|$. 
By Lemma \ref{lem:durdo},
\[\left|B\right| =
\left|\left(A^2\right)_{(\alpha_1,\dotsc,\alpha_k)} \right| \geq \frac{|A|}{n^k}.
\]
Thus, if $k\leq (\log_n |A|)/2$, we see that $|B|\geq \sqrt{|A|}$, and so
\[\left|A^{32 k m + 17}\right|\geq
\frac{1}{n}
|A|^{1 + \frac{1}{20 \log n}} \geq |A|^{1 + \frac{1}{21 \log n}},\]
say, yielding a strong version of conclusion (\ref{eq:uru}).
Assume from now on that $k> (\log_n |A|)/2$.
%Recall that $\log |A| \geq C (\log n)^3$, and so
%$k> C (\log n)^2/2$.

%By definition (\ref{eq:almagest}),

%\left|A'\right|\geq 

%Recall as well that
%$\left|A' \right| \geq \rho^{k-1} (k-1)!$ and 
%$H =\langle A'\rangle$.
Since (\ref{eq:likemula})
holds for $j=k$, and
since we can assume that $\rho n>1$,
there is at least one non-trivial element of $(A^4)_{(\Sigma)}$.
Call it $g_0$. If $\supp(g_0)\leq n/4$, then, by Lemma \ref{lem:bbs},
\[\diam(\Gamma(\Sym(\Omega),A))\leq 4
\diam(\Gamma(\Sym(\Omega),(A^4)_{(\Sigma)}\cup A)) \ll n^8 (\log n)^c,\]
and we are done.
Assume, then, that $\supp(g_0)>n/4$, and so
$|\supp(g_0)\cap O'|> n/4 - (1-e^{-1/10}) n \geq n/7 \geq |O'|/7$.

Now we can finish the argument in any of two closely related ways.
One way would involve combining Prop.~\ref{prop:siniestro} with Lemma
\ref{lem:tagore}, as we said at the beginning of \S \ref{subs:ellafi}.
However, we will find it simpler to proceed in a way closer to the
procedure explained in
\cite[\S 1.5]{MR3152942}. We apply Prop.~\ref{prop:siniestro} with
$O'$ instead of $\Omega$ and $(A')|_{O'}$ instead of $A$.
%(Theorem \ref{thm:siniestro} is applicable because $H$ acts as
%$\Alt(|O'|)$ or $\Sym(|O'|)$ on $O'$, and thus acts $4$-transitively when
%$|O'|\geq e^{-1/10 n}$ is at least $6$.)
We obtain
$\gamma_i\in (A'\cup (A')^{-1} \cup \{e\})^{n^6}$, $1\leq i\leq \ell$, where
  $\ell=O(7 \log n) = O(\log n)$, and
$g_1,\dotsc,g_{k'} \in (A')^v$, $v = O(n^{10})$, $k'=O(\log \log n)$,
such that, for
\[h = \gamma_1 g_0 \gamma_1^{-1} \cdot \gamma_2 g_0 \gamma_2^{-1} \dotsb
\gamma_\ell g_0 \gamma_\ell^{-1} \in \left(A^{4\ell + 2\ell \cdot 8 (k-1) n^6}\right)_{(\Sigma)},\]
the group
\[\langle h, g_1 h g_1^{-1}, g_2 h g_2^{-1},\dotsc , g_{k'} h g_{k'}^{-1}\rangle\]
acts transitively on $O'$. Write $h_0 = h$, $h_i = g_i h g_i^{-1}$
for $1\leq i\leq {k'}$. Since $h$ fixes $\Sigma$ pointwise and $g_i$
fixes $\Sigma$ setwise, $h_i$ fixes $\Sigma$ pointwise for every
$0\leq i\leq {k'}$. Thus, $h_i \in \left(A^{4\ell + 2\ell \cdot 8 (k-1) n^6 + 2 v}\right)_{(\Sigma)}$.
By the same argument, the map
\[\phi:g\mapsto (g h_0 g^{-1}, g h_1 g^{-1}, \dotsc, g h_{k'} g^{-1})\]
sends $(A')^2\subset (A^{16 k})_{\Sigma}$ to a subset of the 
Cartesian product \[\left(A^{32 k +
  4\ell + 2\ell \cdot 8 (k-1) n^6 + 2 v}\right)_{(\Sigma)}
\times \dotsc \times \left(A^{32 k +
  4\ell + 2\ell \cdot 8 (k-1) n^6 + 2 v}\right)_{(\Sigma)}\;\;\;\;\;\;\;\;
\text{($k'+1$ times).}\]
Moreover,
two elements $g$, $g'$ satisfy $\phi(g)=\phi(g')$ if and only
if $g^{-1} g' h_i (g^{-1} g')^{-1}$ for every $0\leq i\leq k'$, i.e., if
and only if $g^{-1} g'$ lies in $C(\langle h_0, h_1, \dotsc,h_{k'}\rangle)$.

We know that $\langle h_0,h_1,\dotsc,h_{k'}\rangle$ acts transitively on $O'$.
It is easy to show that an
element of the centralizer of a transitive group can have a fixed point
if and only if it is the identity.
Thus, if two distinct
$g,g'\in ((A')^2)_{\alpha_k}$ satisfy $\phi(g) = \phi(g')$, then, since
$g (g')^{-1}$ fixes $\alpha_k$, it must act as the identity on
$O'$. In other words, $g (g')^{-1}$ is a non-identity element of support
of size $\leq n - |O'|\leq \left(1 - e^{-1/10}\right) n \leq n/10$.
We can now apply Lemma \ref{lem:bbs} (Babai-Beals-Seress), and obtain
that
\[\diam(\Gamma(\langle A\rangle, A)) \ll 8 k n^8 (\log n)^4 \ll n^{10}.\]

Assume, then, that the restriction of $\phi$ to $((A')^2)_{\alpha_k}$ is
injective. Then
\[\left|\left(A^{32 k +
4\ell + 2\ell \cdot 8 (k-1) n^6 + 2 v}\right)_{(\Sigma)}\right|^{{k'}+1}\geq \left|
((A')^2)_{\alpha_k}\right|\] and so
\[\left|\left(A^N\right)_{(\Sigma)}\right|\geq
\left|((A')^2)_{\alpha_k}\right|^{\frac{1}{k'+1}} \geq \left(\frac{|A'|}{n}\right)^{\frac{1}{k'+1}}\]
for $N = 32 k + 4\ell + 2 \ell \cdot 8 (k-1) n^6 + 2 v = O\left(n^{10}\right)$.
Since $\Sigma = \{\alpha_1,\dotsc,\alpha_{k-1}\}$ and
$|A'|\geq \rho^{k-1} (k-1)!$, where $\rho =e^{-1/5}$, it follows that
\[\begin{aligned}
\left|\left(A^{2 N}\right)_{(\alpha_1,\dotsc,\alpha_k)}\right|&\geq
\frac{\left|\left(A^N\right)_{(\Sigma)}\right|}{n}
\geq \frac{(|A'|/n)^{\frac{1}{k'+1}}}{n}\\
&\geq \frac{\left(\rho^{k-1} (k-1)!/n\right)^{\frac{1}{k'+1}}}{n} \geq
\frac{(\rho^k k!)^{\frac{1}{k'+1}}}{n^2}.\end{aligned}
\]

Since $B = (A^2)_{(\alpha_1,\dotsc,\alpha_k)}$, we know from Lemma \ref{lem:durdo}
that
\[\left|A^{2 N+1}\right|\geq \frac{\left|\left(A^{2 N}\right)_{(\alpha_1,\dotsc,\alpha_k)}\right|}{|B|} \cdot |A|.\]
Hence either
\begin{equation}\label{eq:udunur}
  |A^{2N+1}|\geq \frac{(\rho^k k!)^{\frac{1}{2 k'+2}}}{n} |A|
  \end{equation} or
$|B|\geq (\rho^k k!)^{1/(2 k'+2)}/n$.
In the latter case, by
(\ref{eq:serdukt}),
\begin{equation}\label{eq:adanar}\left|A^{32 k m + 17}\right|\geq \left(\frac{(\rho^k k!)^{\frac{1}{2 k'+2}}}{n}\right)^{\frac{1}{10 \log n}} \frac{|A|}{n}\gg
(\rho^k k!)^{\frac{1}{20 (k'+1) \log n}} \frac{|A|}{n}
  .\end{equation}
The amount on the right in (\ref{eq:udunur}) is clearly greater than that on
the right in
(\ref{eq:adanar}), so we can focus on bounding the right side of
(\ref{eq:adanar}) from below.

By $\rho=e^{-1/5}$, Stirling's formula, 
and the assumptions that
$\log |A|\geq C (\log n)^3$ (or even just $\log |A|> C (\log n)^2$)
and $k>(\log_n |A|)/2$, 
\[\begin{aligned}
\rho^k k! &\gg \left(\frac{k}{e^{6/5}}\right)^k \geq \left(\frac{\log |A|}{2 e^{6/5}\log n}\right)^{\frac{\log |A|}{2 \log n}}
\geq (\log |A|)^{\frac{\log |A|}{4 \log n}} = |A|^{\frac{\log \log |A|}{4 \log n}}
.\end{aligned}\]
Hence, again by $\log |A|\geq C (\log n)^3$,
\[\frac{(\rho^k k!)^{\frac{1}{10 k' \log n}}}{n}
\geq \frac{|A|^{\frac{\log \log |A|}{O((\log n)^2 \log \log n)}}}{n}
\geq |A|^{\frac{\log \log |A|}{O((\log n)^2 \log \log n)}}.\]

Taking $N' = \max(2 N+1,32 k m + 17) = O\left(n^{10}\right)$,
we conclude that
\[\left|A^{N'}\right| \geq |A|^{1 +  \frac{\log \log |A|}{O((\log n)^2 \log \log n)}}.\]
Theorem \ref{thm:jukuju} is thus proved.

\section{Iteration}
We can now prove a marginally weaker version of Theorem
\ref{thm:pais}.
The reader will notice that the proof we are about to give works for any
    $3$-transitive group $G$, not just for $G=\Alt(n)$ and $G=\Sym(n)$.
    However, by \cite[Cor. to Thm.~A]{Pyb93}, every $3$-transitive and in fact every $2$-transitive group $G$ on $n$ elements that is not $\Alt(n)$ or $\Sym(n)$ has
    $\exp(O((\log n)^3))$ elements. Thus, in such a case,
    the result we are about to prove
    would be trivial.
\begin{thm}\label{thm:molop}
  Let $G=\Alt(n)$ or $\Sym(n)$. Let $S$ be a set of generators of $G$. Then
  \begin{equation}\label{eq:klopo}
    \diam \Gamma(G,S) \leq e^{K (\log n)^4 (\log \log n)^2},\end{equation}
  where $K$ is an absolute constant.
\end{thm}
Since $|\Alt(n)|\geq n!/2 \gg (n/e)^n$, it follows immediately that, for
$G= \Alt(n)$ and for $G=\Sym(n)$,
  \[\diam \Gamma(G,S) \leq e^{O((\log \log |G|)^4 (\log \log \log |G|)^2)},\]
  where the implied constant is absolute.
  \begin{proof}
    We can assume that $e\in S$. By Lemma \ref{lem:soti}, we can assume that
  $S=S^{-1}$ as well.
 
  For any $k\geq 1$, if $S^k=S^{k+1}$, then $S^k=S^{k'}$ for every $k'>k$, and so
  $S^k = \langle S \rangle=G$.
  So, if $S^k\ne G$, $\left|S^{k+1}\right|\geq \left|S^{k}\right|+1$.
  Applying this statement for $k=1,2\dotsc,m$, we see that, for any $m$,
  $\left|S^m\right|\geq \min(m,|G|)$. Let $A_0 = S^m$ for
  $m = \lceil \exp\left(C (\log n)^3\right)\rceil$, where $C$ is as in the statement of
  Thm.~\ref{thm:jukuju}. Then, assuming $n$ is larger than a constant,
  $|A_0|\geq \exp\left(C (\log n)^3\right)$. (If $n$ is not larger
  than a constant, then the theorem we are trying to prove is trivial.)

  We apply Theorem \ref{thm:jukuju} to $A_0$ instead of $A$. If
  conclusion (\ref{eq:rororo})
  holds, we stop. If conclusion (\ref{eq:uru})
  holds, we let $A_1 = A_0^{n^C}$ and apply Theorem \ref{thm:jukuju} to $A_1$.
  We keep on iterating until 
  conclusion \ref{eq:rororo} holds, and then we stop.
  We thus have $A_0, A_1,\dotsc, A_k$, $k\geq 0$,
  such that $A_{i+1}=A_i^{n^C}$ for $0\leq i \leq k-1$,
  \begin{equation}\label{eq:sturu}
  \left| A_{i+1}\right|\geq \left|A_i\right|^{1+c \frac{\log \log |A_i|
    }{(\log n)^2 \log \log n}}\end{equation}
  (i.e., conclusion (\ref{eq:uru}) holds) for $0\leq i \leq k-1$,
  and conclusion (\ref{eq:rororo}) holds for $A_k$.

  Let us bound $k$. Write $r_i= \log |A_i|$. By (\ref{eq:sturu}),
  \[r_{i+1} = \left(1+c \frac{\log r_i
  }{(\log n)^2 \log \log n}\right) r_i.\]
  We also know that $r_0\geq 2$ (or really rather more) and $r_k\leq \log |G|$.
  The number of steps needed for $r_i$ to double is
  \[\leq \left\lceil \frac{(\log n)^2 \log \log n}{c \log r_i}
  \right\rceil \leq
  \frac{2 (\log n)^2 \log \log n}{c \log r_i},\]
  where we use the fact that $\lceil y\rceil\leq 2 y$ for $y\geq 1$ and
  we assume, as we may, that $c\leq 1$. We conclude that
  $k$ is at most $(2/c) (\log n)^2 \log \log n$  times
  \[\begin{aligned}\mathop{\sum_{r=2^j}}_{r\leq \log |G|} \frac{1}{\log r} &=
  \sum_{0\leq j\leq \log_2 \log |G|} \frac{1}{j \log 2} \\ &\ll \log \log \log |G|
  \ll \log \log n.\end{aligned}\]
  Write this bound in the form
  \[k\leq C' (\log n)^2 (\log \log n)^2.\]
  We see that $A_k\subset A_0^{n^{C k}} = A_0^{l} = S^{l m}$ for
  $l\leq \exp\left(C C' (\log n)^3 (\log \log n)^2\right)$ and, as before,
  $m\leq \exp(C (\log n)^3)$.
  
  (The author would like to thank L. Pyber profusely for pointing out that,
  as we have just seen, the presence of
  $\log \log |A|$ in the exponent in (\ref{eq:uru}) means we save a factor of
  $(\log n)/\log \log n$ in the bound on $k$.)

  By conclusion (\ref{eq:rororo}), which holds for $A_k$,
  \[\diam(\Gamma(G,S))\leq l m\cdot
  \diam(\Gamma(G,A_k))\leq l m 
  n^C \diam(G'),\]
  where $G'$ is a transitive group on $n'\leq n$ elements
  such that either (a) $n'\leq e^{-1/10} n$ or (b) $G'\nsim \Alt(n'), \Sym(n')$.
  If $n'\leq e^{-1/10} n$ and either $G\sim \Alt(n')$ or $G\sim \Sym(n')$, then
  \begin{equation}\label{eq:rodor1}
    \diam(G')\leq \max(\diam(\Sym(n')),\diam(\Alt(n')))\leq
  4 \diam(\Alt(n'))\end{equation}
  by Lemma \ref{lem:schreier}.
  If $G\nsim \Alt(n'),\Sym(n')$, then, we apply Prop.~\ref{prop:finbo},
  and obtain that
  \begin{equation}\label{eq:rodor2}
    \diam(G)\leq (n')^{C'' \log n'} \prod_{i=1}^k \diam(\Alt(m_i))
\leq e^{C'' (\log n)^2} \prod_{i=1}^k \diam(\Alt(m_i)),
\end{equation}
where 
$\prod_{i=1}^k m_i \leq n'\leq n$,
$m_i\leq n'/2\leq n/2$ for every $1\leq i\leq k$,
and $C''$ is an absolute constant.
Clearly, $l m n^C \max(4,e^{C'' (\log n')^2}) \leq e^{C''' (\log n)^3}$ for
$C'''$ an absolute constant, provided that (say) $n\geq e^{3/2}$.

  We can assume, as an inductive hypothesis,
  that Theorem \ref{thm:molop} is true for
  $G_1=\Alt(n_1)$, $n_1\leq e^{-1/10} n$. In other words,
  \[\diam(G_1) \leq e^{K (\log n_1)^4 (\log \log n_1)^2}.\]
  
  If (\ref{eq:rodor1}) above applies, we let $n_1 = n'$, and obtain that
  \[\begin{aligned}
  \diam(\Gamma(G,S)) &\leq e^{C''' (\log n)^3 (\log \log n)^2}
  e^{K (\log n')^4 (\log \log n')^2}\\
  &\leq e^{\left(C''' (\log n)^3 + K ((\log n)-1/10)^4\right) (\log \log n)^2}.\end{aligned}\] 
  For $K>(10/3.99) \cdot C'''$ (say) and $n$ larger than a constant,
  \[C''' (\log n)^3 + K \left((\log n)-\frac{1}{10}\right)^4 \leq
  K (\log n)^4,\]
  and so Theorem \ref{thm:molop} is true for $n$.

  If (\ref{eq:rodor2}) applies instead, then
  \[
  \diam(\Gamma(G,S)) \leq e^{C''' (\log n)^3 (\log \log n)^2}
  \prod_{i=1}^k e^{K (\log m_i)^4 (\log \log m_i)^2},\]
  where $\prod_{i=1}^k m_i \leq n$ and $m_i\leq n/2$ for all $1\leq i\leq k$.
  If $k=1$, we proceed as above, with $1/2$ instead of $e^{-1/10}$.
  If $k>1$, then, assuming, as we may, that $m_1\geq m_i$ for all
  $2\leq i\leq k$,
  \[\begin{aligned}
      \sum_{i=1}^k (\log m_i)^4 &= (\log m_1)^4 +
  \sum_{i=2}^k (\log m_i)^4 \leq (\log m_1)^4 + (\sum_{i=2}^k \log m_i)^k\\
  &= (\log m_1)^4 + (\log \prod_{i=2}^k m_i)^4 \leq (\log m_1)^4 +
  (\log n - \log m_1)^4\\ &\leq (\log 2)^4 + (\log n/2)^4,\end{aligned}\]
  since $2\leq m_1\leq n/2$. Hence, much as above,
  \[
  \diam(\Gamma(G,S)) \leq e^{C''' (\log n)^3 (\log \log n)^2}
  e^{K ((\log n - \log 2)^4 + (\log 2)^4) (\log \log n)^2}.\]
  For $K>1/(3.99 \log 2)$ and $n$ larger than a constant,
  \[C''' (\log n)^3 + K (\left((\log n)-\log 2\right)^4 + (\log 2)^4) \leq
  K (\log n)^4,\]
  and so Theorem \ref{thm:molop} is true for $n$ in this case as well.
  \end{proof}

\bibliographystyle{alpha}
\bibliography{standrews}
\end{document}